\documentclass[10pt]{article}
\usepackage{amsmath, amscd}
\usepackage{amsfonts}
\usepackage{version}
\usepackage{amssymb}
\usepackage{amsthm}

\newtheorem{theorem}{Theorem}[section]
\newtheorem{lemma}[theorem]{Lemma}
\newtheorem{corollary}[theorem]{Corollary}
\newtheorem{proposition}[theorem]{Proposition}

\newtheorem{question}[theorem]{Question}

\theoremstyle{definition}
\newtheorem{definition}[theorem]{Definition}
\newtheorem{assumption}[theorem]{Assumption}

\theoremstyle{remark}
\newtheorem{remark}[theorem]{Remark}

\newcommand{\B}{\mathbb{B}}
\newcommand{\C}{\mathbb{C}}

\newcommand{\U}{\mathbb{U}}

\newcommand{\ov}{\overline}
\newcommand{\z}{\zeta}

\def\sqr#1#2{{\,\vcenter{\vbox{\hrule height.#2pt\hbox{\vrule width.#2pt
height#1pt \kern#1pt\vrule width.#2pt}\hrule height.#2pt}}\,}}

\begin{document}

\title{Support points for families of univalent mappings on bounded symmetric domains}

\author{HAMADA Hidetaka$^{1,}$\footnote{Corresponding author} \&
KOHR Gabriela$^2$}

\date{}

\maketitle

\begin{center}
{\it $^1$Faculty of Science and Engineering,
Kyushu Sangyo University,
3-1 Matsukadai, 2-Chome, Higashi-ku, Fukuoka 813-8503,
Japan;\\ h.hamada@ip.kyusan-u.ac.jp}
\\
{\it $^2$Faculty of Mathematics and Computer Science,
Babe\c{s}-Bolyai University, 1 M. Kog\u{a}l\-niceanu Str.,
400084 Cluj-Napoca, Romania;\\
gkohr@math.ubbcluj.ro}
\end{center}

\begin{abstract}
In this paper we study some extremal problems for the family $S_g^0(\mathbb{B}_X)$ of
normalized univalent mappings with $g$-parametric
representation on the unit ball $\mathbb{B}_X$ of an $n$-dimensional JB$^*$-triple $X$
with $r\geq 2$, where $r$ is the rank of $X$ and
$g$ is a convex (univalent) function on the unit disc $\mathbb{U}$, which satisfies
some natural assumptions. We obtain sharp coefficient bounds for the family
$S_g^0(\mathbb{B}_X)$, and examples of bounded support points for various
subsets of $S_g^0(\mathbb{B}_X)$.
Our results are generalizations to bounded symmetric domains of known
recent results related to support points for families of univalent mappings
on the Euclidean unit ball $\mathbb{B}^n$ and the unit polydisc $\mathbb{U}^n$ in $\mathbb{C}^n$.
Certain questions will be also mentioned. Finally, we point out sharp 
coefficient bounds and bounded support points for the family $S_g^0(\mathbb{B}^n)$ and 
for special compact subsets of $S_g^0(\mathbb{B}^n)$, in the case $n\geq 2$.
\end{abstract}

{\bf Keywords}{
bounded symmetric domain,
Carath\'eodory family,
$g$-Loewner chain,
parametric representation,
starlike mapping,
support point.}

{\bf MSC(2010)}
{Primary 32H02; Secondary 32M15, 30C55}

\begin{quote}
\textit{We dedicate this paper to Ian Graham, our collaborator and friend,
on the occasion of his 70th birthday}
\end{quote}


\section{Introduction}
\label{sec:intro}
\setcounter{equation}{0}
The study of the family $S(B)$ of normalized univalent
mappings on the unit ball $B$ in $\C^n$ ($n\geq 2$)
with respect to a norm on $\C^n$,
was initiated by H. Cartan
(see \cite{Car}). He was concerned with univalent mappings on the unit
bidisc $\U^2$ in $\C^2$, and gave a counterexample which shows
that the growth theorem for the family $S$ of normalized univalent functions
on the unit disc $\U$ fails in dimension $n=
2$ for $S(\U^2)$. Cartan also suggested that the families of normalized starlike and
convex mappings in $\C^n$ might be considered for further development.
Actually, various
results regarding the family $S$ (growth, covering, coefficient bounds,
distortion, and extremal properties) have extensions to higher
dimensions for these families of univalent mappings on $\B^n$ (see \cite{Gon2}, \cite{Su}).

Loewner chains and the Loewner differential equation in $\C^n$ were first studied by Pfaltzgraff
on the Euclidean unit ball $\B^n$ (see \cite{Pf74}).
He generalized to higher dimensions the Loewner differential equation and developed
existence and uniqueness theorems for its solutions on $\B^n$. Poreda (\cite{Por1}, \cite{Por2})
introduced the family $S^0(\U^n)$ of normalized univalent mappings  with parametric representation on the unit polydisc $\U^n$ in $\C^n$,
and obtained sharp growth and a coefficient bound for $S^0(\U^n)$ under certain additional assumptions.
Graham, Hamada and Kohr \cite{GHK02} introduced and developed the study of the family $S^0(\B^n)$
of mappings with parametric representation on $\B^n$, and obtained growth and coefficient bounds for
some important subsets of $S^0(\B^n)$. We remark that a normalized univalent mapping $f$ on $\B^n$ belongs to
$S^0(\B^n)$ if and only if there is a Loewner chain $f(z,t)$ such
that $\{e^{-t}f(\cdot,t)\}_{t\geq 0}$ is a normal family on $\B^n$ and $f=f(\cdot,0)$ (see \cite{GHK02} and
\cite[Chapter 8]{GK03}). Also, note that in contrast with the family $S(\B^n)$ which is not locally uniformly
bounded in dimension $n\geq 2$, the
family $S^0(\B^n)$ is a compact set with respect to the topology of locally uniform convergence in
the family $H(\B^n)$ of holomorphic mappings from $\B^n$ to $\C^n$.

The existence and regularity theory, as well as many significant differences between the
Loewner theory in  one complex variable and that in higher dimensions have been obtained in the last recent years
(see e.g. \cite{Br}, \cite{BCM}, \cite{BGHK-Variation}, \cite{DGHK}, \cite{GHK02}, \cite{GHK18},
\cite{GHKK13}, \cite{Vo}; see also \cite[Chapter 8]{GK03}, and the references therein).


Graham, Hamada and Kohr \cite{GHK02} introduced the subset $S_g^0(\B^n)$ of $S^0(\B^n)$ which consists
of mappings which have $g$-parametric representation on $\B^n$ (see Definition \ref{d.g-param}),
where $g:\U\to\C$ is a univalent function on $\U$ such that
$g(0)=1$, $\Re g(\zeta)>0$, $\zeta\in \U$, and $g$ satisfies some assumption on $\U$.
These mappings may be embedded as the first elements of $g$-Loewner chains (see Definition \ref{d.g-Loewner}).
Sharp growth and coefficient bounds for the
family $S_g^0(\B^n)$ were obtained in \cite{GHK02} and \cite{GHKK16}. In particular, these results provide also sharp growth
and coefficient estimates for various families of normalized univalent mappings on $\B^n$, such as starlike
and convex mappings (see also \cite{GK03}, and the references therein).

One of the most important problems in the theory of univalent mappings in higher dimensions
is related to the characterization of extreme points and support points for the
family $S^0(\B^n)$. In the case of one complex variable, every extreme/support point of the
family $S$ is an unbounded univalent function (see e.g. \cite[Chapter 6]{Pom2}). In dimension
$n\geq 2$, Bracci \cite{Br} developed a shearing process and found an example of a normalized bounded starlike
mapping on $\B^2$,
which is a support point of the family $S^0(\B^2)$ with respect to a certain linear functional on
$H(\B^2)$. His result is unexpected and very important in the theory of univalent mappings on the unit
ball in $\C^n$, since it provides a basic difference between the one variable theory and that in higher dimensions.

Recently, Graham, Hamada, and Kohr \cite{GHK18} have also considered the family
of mappings with $g$-parametric representation on the unit bidisc $\U^2$ in $\C^2$ (see also \cite{S2},
in the case $g(\zeta)=\frac{1-\zeta}{1+\zeta}$, $\zeta\in\U$). Various sharp coefficient bounds for $S_g^0(\U^2)$
were obtained in \cite{GHK18}.
Also, there are examples of bounded and unbounded support point for the family
$S_g^0(\U^n)$ (see
\cite{GHK18} and \cite{S2}).
Other coefficient bounds for univalent mappings in several complex variables are
given by
\cite{LLX15},
\cite{XL09SC}
and
\cite{XLL18}.


In this paper, we consider extremal problems
for the family $S_g^0(\B_X)$ of mappings with $g$-parametric representation
and for the family of $S_g^*(\B_X)$ of $g$-starlike mappings
on the unit ball $\B_X$ of an $n$-dimensional JB$^*$-triple $X$ of rank $r\geq 2$,
where $g:\U\to \C$ satisfies the conditions of Assumption $\ref{assump}$.
\begin{assumption}
\label{assump}
Let $g:\U\to{\mathbb{C}}$ be a convex (univalent) holomorphic
function such that $g(0)=1$ and $\Re
g(\zeta)>0$, for all $\zeta\in \U$.
\end{assumption}

The family $S_g^*(\B^n)$ of $g$-starlike mappings is a subset of $S^0_g(\B^n)$.
Many important subsets of the family $S^*(\B^n)$ of starlike mappings
are equal to $S_g^*(\B^n)$, for various choices of the function $g$
which satisfies Assumption \ref{assump} and the relation $a_0(g)={\rm dist}(1,\partial g(\U))$
(see \cite{GHK02}).
For example, the family of starlike mappings of order $1/2$ is equal to the family
$S_g^*(\B^n)$, where $g(\zeta)=1-\zeta$, $\zeta\in\U$.
On the other hand, the family
$S_\gamma^*(\B^n)$ of starlike
mappings of order $\gamma\in [0,1)$, is equal to $S_g^*(\B^n)$, where
$g(\zeta)=\frac{1-\zeta}{1+(1-2\gamma)\zeta}$, $\zeta\in \U$.
The family
$SS_\gamma^*(\B^n)$ of strongly starlike
mappings of order $\gamma\in (0,1]$, is equal to $S_g^*(\B^n)$, where
$g(\zeta)=\left(\frac{1-\zeta}{1+\zeta}\right)^{\gamma}$, $\zeta\in \U$,
where the branch of the power function is chosen such
that $\left(\frac{1-\zeta}{1+\zeta}\right)^{\gamma}\Big|_{\zeta=0}=1$.

Graham, Hamada, Kohr and Kohr \cite{GHKK16} generalized
Bracci's result \cite{Br} to the case of mappings in $S_g^0(\B^2)$, and proved the existence of
a bounded $g$-starlike mapping in the family $S_g^0(\B^2)$
of mappings with $g$-parametric representation on the Euclidean unit ball $\B^2$ in $\C^2$
which is a support point for a linear functional.
In the case $g$ satisfies the conditions of Assumption
$\ref{assump}$,
$g(\overline{\zeta})=\overline{g(\zeta)}$ for all $\zeta\in \U$,
$g'(0)<0$
and
$a_0(g)={\rm dist}(1,\partial g(\U))$,
where $a_0(g)$ is given by
\begin{equation}
\label{tilde-a0}
a_0(g)=\inf_{\rho\in (0,1)}\left\{\frac{\min\{|1-g(\rho)|, |g(-\rho)-1|\}}{\rho}\right\},
\end{equation}
then the result contained in \cite[Theorem 4.11 and Remark 4.14]{GHKK16} is as follows.
Note that if $g$ satisfies the above conditions, then
we have (see e.g. \cite{HS} and \cite{MaMi})
\[
\min\{\Re g(\zeta): |\zeta|=\rho\}=g(\rho),\,
\max\{\Re g(\zeta):|\zeta|=\rho\}=g(-\rho),\,\forall\, \rho\in (0,1).
\]

\begin{theorem}
\label{GHKK17-thm}
Let
$g:\U\to\C$
satisfy the conditions of Assumption
$\ref{assump}$,
$g(\overline{\zeta})=\overline{g(\zeta)}$ for all $\zeta\in \U$,
$g'(0)<0$
and
$a_0(g)={\rm dist}(1,\partial g(\U))$.
Also, let $f=(f_1,f_2)\in S_g^0(\B^2)$.
Then
\[
\left|\frac{1}{2}\frac{\partial^2 f_1}
{\partial z_2^2}(0)\right|
\leq \frac{3\sqrt{3}}{2}a_0(g).
\]
This estimate is sharp.
Moreover, $F\in S_g^*(\B^2)$ is a support point for $S_g^0(\B^2)$,
where
\[F(z)=\bigg( z_1+\frac{3\sqrt{3}}{2}a_0(g)z_2^2, z_2 \bigg),
\quad
z=(z_1,z_2)\in \B^2.
\]
Thus, $F$ is also a support point for $S_g^*(\B^2)$.
\end{theorem}

Further, Graham, Hamada and Kohr \cite{GHK18} considered the analog of the
above result on the unit bidisc $\U^2$ in $\C^2$, and proved the existence of
a bounded $g$-starlike mapping in the family $S_g^0(\U^2)$
which is a support point for a linear functional.
In the case $g$ satisfies the above conditions, then the result contained
in \cite[Theorem 5.2 and Remark 5.5]{GHK18} is as follows:

\begin{theorem}
\label{GHKK18-thm}
Let
$g:\U\to\C$
satisfy the conditions of Assumption
$\ref{assump}$,
$g(\overline{\zeta})=\overline{g(\zeta)}$ for all $\zeta\in \U$,
$g'(0)<0$
and
$a_0(g)={\rm dist}(1,\partial g(\U))$.
Also, let $f=(f_1,f_2)\in S_g^0(\U^2)$.
Then
\[
\left| \frac{1}{2}\frac{\partial^2 f_1}
{\partial z_2^2}(0)\right|
\leq a_0(g).
\]
This estimate is sharp.
Moreover, $F\in S_g^*(\U^2)$ is a support point for $S_g^0(\U^2)$, where
\[F(z)=\big( z_1+a_0(g)z_2^2, z_2 \big),\quad z=(z_1,z_2)\in \U^2.\]
Thus, $F$ is also a support point for $S_g^*(\U^2)$.
\end{theorem}

On the other hand, Graham, Hamada, Kohr, and Kohr \cite{GHKK15} generalized Bracci's shearing
result to the case of the family of normalized
biholomorphic mappings on $\B^2$, which have parametric representation with
respect to a diagonal matrix $A={\rm diag}\,(1,\lambda)$, where $\lambda\in [1,2)$, and proved the
existence of a bounded spirallike mapping with respect to $A$, which is a support point for
the above family. Another generalization of Bracci's shearing process to the
case of mappings with parametric representation on $\B^2$ with respect to a time-dependent
linear operator was considered in \cite{HIK}.

The above results are in contrast to the case $n=1$, where
all support points for the family $S=S^0(\U)$ are unbounded.
Also, it is interesting that the coefficient bounds on $\U^2$
are different from those in the case of the Euclidean
unit ball $\B^2$.
Note that
the Euclidean unit ball $\B^n$ and the unit polydisc $\U^n$
are bounded symmetric domains in $\C^n$.
From the point of view of the Riemann mapping theorem, a homogeneous unit ball of a complex Banach
space is a natural generalization of the open unit disc.
Every bounded symmetric domain in
a complex Banach space is biholomorphically equivalent to a homogeneous unit ball.

Based on the above arguments, it is therefore of
interest to consider $g$-Loewner chains and mappings which have $g$-parametric
representation on the Euclidean unit ball $\B^n$ and the unit polydisc $\U^n$,
as well as on other bounded symmetric domains in $\C^n$ which contain the origin,
and to study if extremal problems for the families
$S_g^0(\B^n)$ and $S_g^0(\U^n)$
may be extended in the case of bounded symmetric domains in $\C^n$.
Also, there exists a mapping $g$ on $\U$ which satisfies the conditions of Assumption $\ref{assump}$
and $a_0(g)>{\rm dist}(1,\partial g(\U))$ (see \cite[Remark 4.4]{GHKK16}).
Thus, it is natural to ask
the following questions:

\begin{question}
\label{q.01}
Let $g:\U\to\C$ satisfy the conditions of Assumption $\ref{assump}$.
Can we obtain bounded support points of the families $S_g^0(\B^2)$
and $S_g^0(\U^2)$
without assuming $a_0(g)={\rm dist}(1,\partial g(\U))$?
\end{question}

\begin{question}
\label{q.02}
Let $g:\U\to\C$ satisfy the conditions of Assumption $\ref{assump}$.
Can we obtain bounded support points of the family $S_g^0(D)$
for other bounded symmetric domains $D$ in $\C^n$, $n\geq 2$?
\end{question}

In this paper, we give positive answers to the above questions
in the case the domain $D$ is a bounded symmetric domain
realized as the unit ball of an $n$-dimensional JB$^*$-triple $X=(\C^n, \| \cdot\|_X)$
of rank $r\geq 2$.
For the proof, we use the convexity of $g$ and also the property that 
there exists an $r$-dimensional subspace $X_1$ of $\C^n$ such that
$\B_X\cap X_1$ may be regarded as the unit polydisc $\U^r$ of dimension $r$.
The main results complement recent extremal results for mappings
with parametric representation on the unit ball $\B^n$
(see \cite{Br}, \cite{BR}, \cite{GHKK15},
\cite{GHKK16}, \cite{HIK}), and on the unit polydisc $\U^n$ in $\C^n$ (see
\cite{GHK18}, \cite{S2}).
Since we do not assume that $a_0(g)={\rm dist}(1,\partial g(\U))$,
our result is an improvement of the above theorems.
As a new corollary, we obtain a sharp coefficient bound and bounded support points
for the family of strongly starlike mappings of order $\alpha$ on $\B_X$.
Note that, in general, we have
$a_0(g)\geq {\rm dist}(1,\partial g(\mathbb{U}))$
for functions $g$ which satisfy the conditions of Assumption $\ref{assump}$
(see Proposition \ref{p.a0}).

Let $r$ be the rank of $X$
and let $e_1, \ldots, e_r$ be a frame of $X$.
There exist $e_{r+1}, \ldots, e_n\in X$ such that $e_1, \ldots, e_n$ is an orthogonal basis of $X$
with respect to the Bergman metric $h_0$ on $\B_X$ at $0$.
For $z=z_1e_1+\cdots +z_n e_n\in X$,
we use the notation $z=(z_1, z_2, \dots, z_n)$.

The main results of this paper are given below. The notations will be explained in
the next sections.

\begin{theorem}
\label{t0.1}
Let $\B_X$ be  the unit ball of an $n$-dimensional JB$^*$-triple $X=(\C^n, \| \cdot\|_X)$
of rank $r\geq 2$.
Let
$g:\U\to\C$
satisfy the conditions of Assumption
$\ref{assump}$
and let $d_1(g)={\rm dist}(1,\partial g(\U))$.
Also, let $f=(f_1,\dots, f_n)\in S_g^0(\B_X)$.
Then
\[
\bigg|\frac{1}{2}\frac{\partial^2 f_i}
{\partial z_j^2}(0)\bigg| \leq d_1(g),
\quad \mbox{for }
1\leq i\neq j\leq r.
\]
These estimates are sharp.
Moreover, $F_{i,j}[g]\in S_g^*(\B_X)$ $(1\leq i\neq j\leq r)$
are bounded support points for
$S_g^0(\B_X)$,
where $F_{i,j}[g](z)=z\pm d_1(g)z_j^2e_i,\quad z\in\B_X$.

In particular,
$F_{i,j}[g]$ $(1\leq i\neq j\leq r)$
are bounded support points for
$S_g^*(\B_X)$.
\end{theorem}

\begin{theorem}
\label{t0.3}
Let
$g:\U\to\C$
satisfy the conditions of Assumption
$\ref{assump}$
and let
$d_1(g)={\rm dist}(1,\partial g(\U))$.
Also, let $n\geq 2$ and let $f=(f_1,\dots, f_n)\in S_g^0(\B^n)$.
Then
$$\left|
\frac{1}{2}\frac{\partial^2 f_i}{\partial z_j^2}(0)
\right|
\leq \frac{3\sqrt{3}}{2}d_1(g),
\quad
\mbox{for }
1\leq i\neq j\leq n.$$
These estimates are sharp. Moreover,
$\widetilde{F}_{i,j}[g]\in S_g^*(\B^n)$ are bounded support points for $S_g^0(\B^n)$,
$(1\leq i\neq j\leq n)$, where
$\widetilde{F}_{i,j}[g](z)=z\pm\frac{3\sqrt{3}}{2}d_1(g)z_j^2e_i$, $z\in\B^n$. 

In particular, $\widetilde{F}_{i,j}[g]$ $(1\leq i\neq j\leq r)$ 
are bounded support points for $S_g^*(\B^n)$.
\end{theorem}

\begin{theorem}
\label{t0.2}
Let $B$ be the unit ball of $\C^n$ with respect to an arbitrary norm on $\C^n$.
Let $g:\U\to\C$ be a univalent holomorphic function such that $g(0)=1$,
$g(\overline{\zeta})=\overline{g(\zeta)}$, 
and 
$\Re g(\zeta)>0$, $\zeta\in\U$. 
Assume that
$g(\rho)=O(1-\rho)$ as $\rho\to 1-0$.
Then there exists an unbounded support point for $S_g^0(B)$.
\end{theorem}

\section{Preliminaries}
\setcounter{equation}{0}
First, we give the relation of $a_0(g)$ and ${\rm dist}(1,\partial g(\U))$ for functions $g$ which satisfy
the conditions of Assumption \ref{assump}.

\begin{proposition}
\label{p.a0}
Let $g:\U\to\C$ be a convex $($univalent$)$ function, which satisfies the conditions of Assumption $\ref{assump}$.
Also, let $a_0(g)$ be given by $(\ref{tilde-a0})$. Then we have
$a_0(g)\geq {\rm dist}(1,\partial g(\mathbb{U}))$.
\end{proposition}
\begin{proof}
Let $a_0=a_0(g)$.
We note that the relation $(\ref{tilde-a0})$ is equivalent to
$a_0=\inf_{r\in(-1,1)}\left|h(r)\right|$, where $h:\U\to\C$ is given by
\[
h(\z)=
\left\{
\begin{array}{ll}
\frac{g(\z)-1}{\z},&\z\in\U\setminus\{0\}
\\
g'(0),&\z=0.
\end{array}
\right.
\]
Since $g(\mathbb{U})$ is convex, we deduce, by \cite[Proposition 5.6]{Pom1},
that $g$ admits a continuous extension $g:\overline{\mathbb{U}}\to \mathbb{C}
\cup\{\infty\}$. Then also $h$ admits a continuous extension $h:\overline{\mathbb{U}}\to \mathbb{C}
\cup\{\infty\}$ and $h(\z)\neq 0$, $\z\in {\U}$, since $g$ is univalent.
Using the maximum principle to $1/h$, we deduce:
\begin{eqnarray*}
{\rm dist}(1,\partial g(\mathbb{U}))&=&\inf_{\z\in\partial\mathbb{U}}|g(\z)-g(0)|
\\
&=&
\inf_{\z\in\partial\mathbb{U}}|h(\z)|=
\inf_{\z\in\ov{\U}}|h(\z)|
\\
&\leq&
 \inf_{r\in(-1,1)}|h(r)|=a_0.
\end{eqnarray*}
Hence ${\rm dist}(1,\partial g(\mathbb{U}))\le a_0$, as claimed.
This completes the proof.
\end{proof}

Let $X$ and $Y$ be complex Banach spaces.
Let $L(X,Y)$ denote the family of continuous linear operators from $X$ to $Y$.
The family $L(X,X)$ is denoted by $L(X)$, and the identity in $L(X)$ is denoted by $I_X$.
Also, let $\B_X$ be the open unit ball of $X$, and let $H(\B_X)$
be the family of holomorphic mappings from $\B_X$ into $X$.
Let $Df(x)$ denote the Fr\'{e}chet derivative of $f\in H(\B_X)$.
A mapping $f\in H(\B_X)$ is said to be normalized if $f(0)=0$ and $Df(0)=I_X$.
Let ${\mathcal L}S(\B_X)$ be the family of normalized locally
biholomorphic mappings $f\in H(\B_X)$, and
let $S(\B_X)$ be the family of normalized biholomorphic mappings $f\in H(\B_X)$.
The family $S(\U)$ is denoted by $S$, where $\U$ is the unit disc in $\C$.
Also, let $S^*(\B_X)$ (resp. $K(\B_X)$) be the subset of $S(\B_X)$ consisting of
starlike mappings (resp. convex mappings) on $\B_X$,
where a mapping $f\in S(\B_X)$ is said to be  starlike
(respectively, convex)
if $f(\B_X)$ is a starlike
(respectively, convex) domain in $X$.

\begin{definition}
(see e.g. \cite{GK03})
\label{d-subordination}
Let $X$ be a complex Banach space and let
$\B_X$ be the open unit ball of $X$.


(i) A mapping $f:\B_X\times [0,\infty )\to X$ is called a Loewner chain
if $e^{-t}f(\cdot,t)\in S(\B_X)$ for $t\geq 0$, and
and $f(\B_X,s)\subseteq f(\B_X,t)$, $0\leq s\leq t<\infty$.

(ii) Let $f=f(z,t):\B_X\times [0,\infty)\to X$ be a Loewner chain.
Then there exists a unique biholomorphic Schwarz mapping $v=v(\cdot,s,t)$,
called the transition mapping associated with $f(z,t)$, such that
$f(z,s)=f(v(z,s,t),t)$ for $z\in\B_X$ and $t\geq s\geq 0$.
Also, $Dv(0,s,t)=e^{s-t}I_X$ for $t\geq s\geq 0$. The family
$(v_{s,t})_{t\geq s}$ is also called the evolution family associated
with $f(z,t)$, where $v_{s,t}(\cdot)=v(\cdot,s,t)$.
\end{definition}

For every $z\in X\setminus\{0\}$, let
$$T(z)=\Big\{l_z\in L(X,\C): \|l_z\|=1, l_z(z)=\|z\|\Big\}.$$
The family $T(z)\neq\emptyset$, in view of the Hahn-Banach theorem.

Next, we recall the Carath\'eodory family ${\mathcal M}={\mathcal M}(\B_X)$ in $H(\B_X)$
(see \cite{Su}):
$$\mathcal{M}(\B_X)=\big\{h\in H(\B_X):h(0)=0, Dh(0)=I_X,$$
$$\Re l_z(h(z))>0,
z\in \B_X\setminus\{0\}, l_z\in T(z)\big\}.$$

If $X=\mathbb{C}$, it is clear that $f\in \mathcal{M}(\U)$ if and only
if $f(z_1)/z_1\in\mathcal{P}$, where
$$\mathcal{P}=\big\{p\in H(\U): p(0)=1, \Re p(z_1)>0, z_1\in \U\big\}$$
is the Carath\'eodory family on the unit disc $\U$.

The family ${\mathcal M}(\B_X)$ occurs in the study of
various problems regarding univalent mappings in $\C^n$ and complex
Banach spaces, as well as in Loewner's theory in higher dimensions
(see \cite{ABW}, \cite{Br}, \cite{BR},
\cite{DGHK}, \cite{ERS}, \cite{GHK02}, \cite{GHKK12}, \cite{GHKK13}, \cite{GK03},
\cite{M18}, \cite{Pf74}, \cite{Por1}, \cite{Por2}, \cite{ReSh}, \cite{Roth-pontryagin},
\cite{S}, \cite{Su}, \cite{Vo}).

\begin{definition}
\label{d.Ag-class}
(see e.g. \cite{GHK02} and \cite{GHK14})
Let $g:\U\to\mathbb{C}$ be a univalent holomorphic function such that $g(0)=1$ and
$\Re g(\z)>0$, $\z\in\U$. Also, let $h:\B_X\to X$ be a
normalized holomorphic mapping. We say that $h$ belongs to the
family ${\mathcal M}_g(\B_X)$ if
\[
\frac{1}{\|z\|}l_z(h(z))\in g(\U), \quad z\in
\B_X\setminus\{0\}, \quad l_z\in T(z).
\]
\end{definition}

\begin{remark}
\label{remark-g}
Since $h$ is normalized,
$g(0)$ should be $1$ so that ${\mathcal M}_g(\B_X)$ is well-defined.
Also, since $\Re g(\z)>0$, $\z\in\U$,
we have
${\mathcal M}_g(\B_X)\subset {\mathcal M}(\B_X)$.
\end{remark}

\begin{definition}
\label{g-starlike} (see \cite{HH08})
Let $g:\U\to\C$ be a univalent holomorphic
function such that $g(0)=1$ and $\Re g(\zeta)>0$, $\z\in\U$.
Also, let $f\in {\mathcal L}S(\B_X)$. The mapping $f$ is said to be
$g$-starlike (denoted by $f\in S_g^*(\B_X)$) if $h\in {\mathcal M}_g(\B_X)$,
where $h(z)=[Df(z)]^{-1}f(z)$, $z\in\B_X$.
\end{definition}

\begin{remark}
\label{r.examples}
Various choices of the function $g$ in the above definition
provide important subsets of ${\mathcal M}(\B_X)$ and $S^*(\B_X)$.

(i) Let $\alpha\in [0,1)$ and $g(\zeta)=\frac{1-\zeta}{1+(1-2\alpha)\zeta}$,
$|\zeta|<1$. In this case, the family ${\mathcal M}_g(\B_X)$ will be
denoted by ${\mathcal M}_\alpha(\B_X)$, while the family $S_g^*(\B_X)$
will be denoted by $S_\alpha^*(\B_X)$. Note that $S_\alpha^*(\B_X)$
is the usual family of starlike mappings of order $\alpha$ on $\B_X$
(see e.g. \cite[Chapter 6]{GK03}). It is
known that $K(\B_X)\subseteq S_{1/2}^*(\B_X)$ (see \cite{RS}).

(ii) Let $\alpha\in [0,1)$ and $g(\zeta)=\frac{1-(1-2\alpha)\zeta}{1+\zeta}$,
$|\zeta|<1$. In this case,
the family $S_g^*(\B_X)$
will be denoted by ${\mathcal A}S_\alpha^*(\B_X)$.
The family ${\mathcal M}_g(\B_X)$ is related to the
family ${\mathcal A}S_\alpha^*(\B_X)$ of almost starlike mappings of order
$\alpha$ on $\B_X$ (see \cite{XL}). More precisely, if $f\in {\mathcal L}S(\B_X)$,
then $f\in {\mathcal A}S_\alpha^*(\B_X)$ if $h\in {\mathcal M}_g(\B_X)$, where
$h(z)=[Df(z)]^{-1}f(z)$, $z\in\B_X$ (see \cite{XL}).
It is clear that ${\mathcal A}S_\alpha^*(\B_X)\subseteq S^*(\B_X)$.

(iii) Let $\alpha\in (0,1]$ and $g(\zeta)=\left(\frac{1-\zeta}{1+\zeta}\right)^{\alpha}$,
$\zeta\in\U$. We choose the branch of the power function such that $g(0)=1$. In this case,
the family $S_g^*(\B_X)$
will be denoted by $SS_\alpha^*(\B_X)$.
The family ${\mathcal M}_g(\B_X)$ is connected with the family $SS_\alpha^*(\B_X)$
of strongly starlike mappings
of order $\alpha$ (see \cite{HH08}). If $f\in {\mathcal L}S(\B_X)$, then $f\in SS_\alpha^*(\B_X)$
if $h\in {\mathcal M}_g(\B_X)$, where $h(z)=[Df(z)]^{-1}f(z)$, $z\in\B_X$ (see \cite{HH08}).
Clearly, $SS_\alpha^*(\B_X)\subseteq S^*(\B_X)$.
\end{remark}

\begin{definition}
(see e.g. \cite{C12}, \cite{K83} and \cite{L})
A complex Banach space $X$ is called a JB$^*$-triple
if $X$ is a complex Banach space equipped with a continuous
Jordan triple product
\[
X\times X\times X \to X \qquad (x,y,z)\mapsto \{ x, y, z\}
\]
satisfying
\begin{enumerate}
\item[$(\mbox{J}_1)$]
$\{ x, y, z\}$ is symmetric bilinear in the outer variables, but conjugate linear in the middle variable,
\item[$(\mbox{J}_2)$]
$\{ a, b, \{ x, y, z\}\}=\{\{a, b, x\}, y, z\}-\{x, \{b, a, y\}, z\}+\{x, y, \{ a, b, z\}\}$,
\item[$(\mbox{J}_3)$]
$x\square x\in {L}(X)$ is a hermitian operator with spectrum $\geqq 0$,
\item[$(\mbox{J}_4)$]
$\| \{x, x, x\}\|=\| x\|^3$
\end{enumerate}
for $a,b,x,y, z\in X$, where the {\it box operator} $x\square y :X \to X$  is defined by
$x\square y( \cdot) = \{ x, y, \cdot\}$ and $\| \cdot \|$ is the norm on $X$.
\end{definition}

An element $u$ in a JB*-triple $ X$
is called a {tripotent} if $\{ u, u, u\}=u$.
Two tripotents $u$ and $v$ are said to be
{orthogonal} if $D(u,v)=0$,
where $D(u,v)=2u \square v$.
Orthogonality is a symmetric relation.
A tripotent $u$ is said to be {maximal} if the only
 tripotent which is orthogonal to $u$ is $0$.
A tripotent $u$ is said to be primitive if it cannot be
written as a sum of two non-zero orthogonal tripotents.
A frame is a maximal family of pairwise orthogonal,
primitive tripotents.
The cardinality of all frames is the same,
and is called the rank $r$ of $X$.

From now on, throughout this paper, we assume that $X$ is a finite dimensional
complex Banach space.

\begin{definition}
(cf. \cite{BCM} and \cite{DGHK})
\label{d.vector-field}
A mapping $h=h(z,t):\B_X\times [0,\infty)\to X$ is called a
generating vector field (Herglotz vector field) if the following
conditions hold:
\begin{enumerate}
\item[(i)]
$h(\cdot,t)\in {\mathcal M}(\B_X)$, for all $t\geq 0$;
\item[(ii)]
$h(z,\cdot)$ is measurable on $[0,\infty)$, for all $z\in\B_X$.
\end{enumerate}
\end{definition}

\begin{remark}
\label{r.loewner-eq}
(i) Let $f=f(z,t):\B_X\times [0,\infty)\to X$ be a Loewner chain. Then there exists a Herglotz
vector field $h=h(z,t):\B_X\times [0,\infty)\to X$
such that $f(z,t)$ is a solution of the Loewner differential equation
(see \cite{GHK02}; see e.g. \cite[Chapter 8]{GK03})
\begin{equation}
\label{loewner}
\frac{\partial f}{\partial t}(z,t)=Df(z,t)h(z,t),\quad \mbox{ a.e. }\quad t\geq 0,\quad
\forall\,z\in \B_X.
\end{equation}

(ii) Conversely, if $h=h(z,t):\B_X\times [0,\infty)\to X$ is a Herglotz vector field, then
every univalent solution $f(z,t)$ of the Loewner differential
equation (\ref{loewner}) is a Loewner chain
(see \cite{DGHK}; see e.g. \cite[Chapter 8]{GK03}).
\end{remark}

Now, we recall the notion of a $g$-Loewner chain and $g$-parametric representation on
the open unit ball $\B_X$ (see \cite{GHK02}; see also \cite{GHK14}).

\begin{definition}
\label{d.g-Loewner}
Let $f=f(z,t):\B_X\times [0,\infty)\to X$ be a normal Loewner chain, i.e. $f(z,t)$ is a Loewner
chain such that $\{e^{-t}f(\cdot,t)\}_{t\geq 0}$ is a normal family on $\B_X$.
Also, let $g:\U\to\C$ satisfy the conditions of Assumption \ref{assump}.
We say that $f(z,t)$ is a $g$-Loewner chain if $h(\cdot,t)\in {\mathcal M}_g(\B_X)$, for a.e.
$t\in [0,\infty)$,
where $h=h(z,t)$ is the Herglotz vector field given by (\ref{loewner}).
\end{definition}

\begin{definition}
\label{d.g-param}
Let $g:\U\to\C$ satisfy the conditions of Assumption \ref{assump}.
Also, let $f\in H(\B_X)$ be a normalized mapping. The mapping
$f$ is said to have $g$-parametric representation (denoted by $f\in S_g^0(\B_X)$) if
there exists
a Herglotz vector field $h:\B_X\times [0,\infty)\to X$ such that
$h(\cdot,t)\in {\mathcal M}_g(\B_X)$ for a.e. $t\geq 0$, and
$f(z)=\displaystyle\lim_{t\to\infty}e^{t}v(z,t)$
locally uniformly
on $\B_X$, where $v(z,t)=v(z,0,t)$ and $v(z,s,t)$ is the unique locally Lipschitz
continuous solution on $[s,\infty)$ of the initial value problem
\begin{equation}
\label{2.1}
\frac{\partial v}{\partial t}=-h(v,t)\quad \mbox{ a.e. }\quad
t\geq s,\quad v(z,s,s)=z,
\end{equation}
for all $z\in \B_X$ and $s\geq 0$.

If $\alpha\in [0,1)$ and
$g(\z)=\frac{1-\z}{1+(1-2\alpha)\z}$, $\z\in\U$, then the family $S_g^0(\B_X)$ will
be denoted by $S_\alpha^0(\B_X)$, while if $g(\z)=\frac{1-(1-2\alpha)\z}{1+\z}$, $\z\in\U$,
then the family $S_g^0(\B_X)$ will be denoted by ${\mathcal A}S_\alpha^0(\B_X)$. Also,
if $\alpha\in (0,1]$ and $g(\zeta)=\left(\frac{1-\zeta}{1+\z}\right)^\alpha$, $\zeta\in\U$,
then the family $S_g^0(\B_X)$ will be denoted by $SS_\alpha^0(\B_X)$.
\end{definition}

\begin{remark}
Various choices of the function $g$ in the above definition provides important subsets of
$S^0(\B_X)$.

(i) Assume that $g:\U\to\C$ satisfies the conditions of Assumption \ref{assump}. Then
$S_g^0(\B_X)\subseteq S^0(\B_X)$, and if $g(\zeta)=\frac{1-\zeta}{1+\zeta}$,
$|\zeta|<1$, then $S_g^0(\B_X)=S^0(\B_X)$.

(ii) Let $f\in H(\B_X)$ be a normalized mapping. Then $f\in S_g^0(\B_X)$
if and only if there exists a $g$-Loewner chain $f(z,t)$ such that $f=f(\cdot,0)$
(see \cite[Theorem 1.4, Lemma 1.6]{GHK02}).

(iii) Clearly, $S_g^*(\B_X)\subseteq S_g^0(\B_X)$. Indeed, $f\in S_g^*(\B_X)$ if and
only if $f(z,t)=e^tf(z)$ is a $g$-Loewner chain (\cite{GHK02}; see \cite{Por1}, in the case
$g(\zeta)=\frac{1-\zeta}{1+\zeta}$, $\zeta\in\U$).

(iv) Let $f\in {\mathcal L}S(\B_X)$ and $\alpha\in [0,1)$. Then $f$ is starlike of order $\alpha$ on $\B_X$
if and only if $f(z,t)=e^tf(z)$ is a $g$-Loewner chain, where $g(\zeta)=\frac{1-\zeta}{1+(1-2\alpha)\zeta}$,
for $\zeta\in \U$. Hence $S_\alpha^*(\B_X)\subseteq S_\alpha^0(\B_X)$ (see e.g. \cite{GHK02}).

(v) Let $f\in {\mathcal L}S(\B_X)$ and $\alpha\in [0,1)$. Then $f\in {\mathcal A}S_\alpha^*(\B_X)$
if and only if $f(z,t)=e^tf(z)$ is a $g$-Loewner chain, where $g(\zeta)=\frac{1-(1-2\alpha)\zeta}{1+\zeta}$, $\zeta\in\U$
(cf. \cite{XL}).

(vi) Let $f\in {\mathcal L}S(\B_X)$ and $\alpha\in (0,1]$. Then $f\in SS_\alpha^*(\B_X)$ if and only if
$f(z,t)=e^tf(z)$ is a $g$-Loewner chain, where $g(\zeta)=\left(\frac{1-\zeta}{1+\zeta}\right)^\alpha$, $\zeta\in\U$.
\end{remark}

\begin{definition}
\label{d.support}
Let $E\subseteq H(\B_X)$ be a nonempty set.

(i) A point $f\in E$ is called a support point of $E$
(denoted by $f\in {\rm supp}\,E$) if there is a
continuous linear functional
$\Lambda:H(\B_X)\to \C$ such that $\Re \Lambda|_E\neq$ constant and
$\Re \Lambda(f)=\max\{\Re \Lambda(q): q\in E\}$.

(ii) A point $f\in E$ is called an extreme point of $E$ (denoted by
$f\in {\rm ex}\,E$)
if $f=\lambda g+(1-\lambda)h$, for some $\lambda\in (0,1)$, $g,h\in E$, implies
that $g=h$.
\end{definition}

From now on, throughout this paper,
we assume that  $\B_X$ is the unit ball of an $n$-dimensional JB$^*$-triple $X=(\C^n, \| \cdot\|)$.
Let $h_0$ be the Bergman metric
on the unit ball $\B_X$ of an $n$-dimensional JB$^*$-triple $X=(\C^n, \| \cdot\|)$ at $0$.
We obtain the following lemma (cf. \cite[Lemma 2.1]{H18}).

\begin{lemma}
\label{estimate-trace}
Let $\B_X$ be a bounded symmetric domain
realized as the open unit ball of a JB$^*$-triple $X=(\C^n, \| \cdot\|)$,
and let $e$ be an arbitrary tripotent in $X$.
Then we have
\[
|h_0(x,e)|\leq \| x\| h_0(e,e),
\quad x\in X.
\]
\end{lemma}

\begin{proof}
We may assume that $x\neq 0$ and $e\neq 0$.
By \cite[Lemma 6.2]{L},
we have
\[
\frac{1}{(1+\varepsilon)\| x\|}|h_0(x,e)|<h_0(e,e)
\]
for any $\varepsilon>0$.
Letting $\varepsilon \to 0$, we obtain the lemma.
This completes the proof.
\end{proof}
\begin{remark}
If $r=1$, then $\B_X$ is irreducible.
Let $w_0 \in \partial \B_X$.
Then there exist a frame $\mathbf{u}=(u_1, \dots, u_r)$ of $X$
and constants $1=\lambda_1\geq \lambda_2 \geq \cdots \geq \lambda_r\geq 0$
such that
$w_0=u_1+\lambda_2 u_2+\cdots +\lambda_r u_r$.
Since $\B_X$ is irreducible,
we have
\[
h_0(w_0, w_0)=g\sum_{i=1}^{r}\lambda_i^2,
\]
where $g$ is a constant, called the genus of $X$
(see \cite{H18}).
Then $h_0(u,u)$ is constant on the set of primitive tripotents in $X$.
Also, there exists an orthogonal basis $e_1$, $\dots$, $e_n$ of $X$
with respect to the Bergman metric $h_0$ on $\B_X$ at $0$
such that each $e_j$ is a primitive tripotent in $X$.
Let $z=z_1 e_1+\dots+ z_ne_n \in X\setminus \{ 0\}$.
We will consider the condition for $z\in \B_X$.
Let $z= ce^*$ be the spectral decomposition of $z$.
Then we have
\[
\| z\|^2=c^2=\frac{h_0(z,z)}{h_0(e^*, e^*)}
=\frac{\sum_{j=1}^n|z_j|^2h_0(e_j,e_j)}{h_0(e^*, e^*)}
=\sum_{j=1}^n|z_j|^2.
\]
Thus, $z\in \B_X$ if and only if $\sum_{j=1}^n|z_j|^2<1$,
and we may suppose $\B_X=\B^n$ if $r=1$,
where $\B^n$ is the Euclidean unit ball of $\C^n$.
\end{remark}

In the rest of this paper, we assume that the rank $r$ of $X$ satisfies $r\geq 2$.
Let $e_1, \ldots, e_r$ be a frame of $X$.
There exist $e_{r+1}, \ldots, e_n\in X$ such that $e_1, \ldots, e_n$ is an orthogonal basis of $X$
with respect to the Bergman metric $h_0$ on $\B_X$ at $0$.
For $z=z_1e_1+\cdots +z_n e_n\in X$,
we also use the notation $z=(z_1, z_2, z'')=(z_1,z')$.

From Lemma \ref{estimate-trace}, we obtain the following lemma:

\begin{lemma}
\label{linear1}
For $z=(z_1,z_2,0'')\in X\setminus \{ 0\}$,
let
\[
l^{(1)}_z(w)=\frac{|z_1|}{z_1}w_1,
\quad w\in X,
\]
for $z_1\neq 0$,
and
\[
l^{(2)}_z(w)=\frac{|z_2|}{z_2}w_2,
\quad w\in X,
\]
for $z_2\neq 0$.
Then $l^{(1)}_z\in T(z)$ for $|z_1|=\| z\|$,
and $l^{(2)}_z\in T(z)$ for $|z_2|=\| z\|$.
\end{lemma}

\begin{proof}
Let $w=w_1e_1+\cdots +w_ne_n\in X$.
Since $e_1,\dots, e_n$ are orthogonal to each other with respect to $h_0$,
it follows that $|w_1|h_0(e_1,e_1)=|h_0(w,e_1)|$. Also, since
$|h_0(w,e_1)|\leq \| w\| h_0(e_1,e_1)$ by Lemma \ref{estimate-trace},
we obtain that $|w_1|\leq \|w\|$, and thus we have
$|l_z^{(1)}(w)|\leq \| w\|$.
Also, we have $l_z^{(1)}(z)=|z_1|=\| z\|$
for $|z_1|=\| z\|$.
Therefore, $l_z^{(1)}\in T(z)$ for $|z_1|=\| z\|$.
The proof for $l_z^{(2)}$ is similar.
This completes the proof.
\end{proof}

\section{Coefficient bounds for the family ${\mathcal M}_g(\B_X)$}
\label{sec:coeff}
\setcounter{equation}{0}
Let $\B_X$ be a bounded symmetric domain
realized as the open unit ball of a JB$^*$-triple $X=(\C^n, \| \cdot\|)$
of rank $r\geq 2$.
We begin this section with the following notion, which
is an analogue on $\B_X$ of the shearing process due to Bracci \cite[Definition 1.3]{Br}.
Then we obtain some coefficient bounds for the family ${\mathcal M}_g(\B_X)$, where
$g:\U\to \C$ satisfies the conditions of Assumption \ref{assump}
(cf. \cite[Proposition 2.1]{Br}, \cite[Proposition 4.2]{GHK18}, and
\cite[Proposition 4.2]{GHKK16}).
We also obtain various
particular cases. We shall apply Proposition \ref{p.shear1},
to obtain examples of support points for the family $S_g^0(\B_X)$.

\begin{definition}
\label{d.shearing}
Let $\B_X$ be a bounded symmetric domain
realized as the open unit ball of a JB$^*$-triple $X=(\C^n, \| \cdot\|)$ of rank $r\geq 2$.
Let $h=(h_1,\dots, h_n)\in H(\B_X)$ be such that $h(0)=0$.
Then the shearings $h_{i,j}^{[c]}$ $(1\leq i\neq j\leq r)$ of $h$ are given by
\[h_{i,j}^{[c]}(z)=Dh(0)z+\frac{1}{2}\frac{\partial^2 h_i}{\partial z_j^2}(0)z_j^2 e_i.\]
\end{definition}

\begin{proposition}
\label{p.shear1}
Let $\B_X$ be as in Definition \ref{d.shearing}.
Let $g:\U\to \C$ satisfy the conditions of Assumption
$\ref{assump}$ and let $d_1(g)={\rm dist}(1,\partial g(\U))$.
If $h=(h_1,h_2,\dots, h_n)\in {\mathcal M}_g(\B_X)$, then
\begin{equation}
\label{estim1}
\left|
\frac{1}{2}\frac{\partial^2 h_i}{\partial z_j^2}(0)
\right| \leq  d_1(g),
\quad
\mbox{for }
1\leq i\neq j\leq r.
\end{equation}
In addition, $h_{i,j}^{[c]}\in {\mathcal M}_g(\B_X)$
for $1\leq i\neq j\leq r$.
\end{proposition}

\begin{proof}
We only give a proof in the case $i=1$ and $j=2$.
The other cases can be proved by using similar arguments.
Also, we shall use arguments which modify those in the proofs of \cite[Proposition 2.1]{Br}
and \cite[Proposition 4.2]{GHK18}.

Since $h(0)=0$ and $Dh(0)=I_n$, we deduce that $h$
has the following power series expansion on $\B_X \cap \{ (z_1,z_2,0'')\in \C^n\}$:
$$h(z)=\bigg(z_1+\sum_{\alpha\in \mathbb{N}^2,|\alpha|\geq 2}q_\alpha^1z^\alpha,
z_2+\sum_{\alpha\in \mathbb{N}^2,|\alpha|\geq 2}q_\alpha^2z^\alpha, h''\bigg),
\quad z=(z_1,z_2, 0'')\in\B_X.$$
Since $h\in {\mathcal M}_g(\B_X)$, it follows that
\[
\frac{1}{\|z\|}l_z(h(z))\in g(\U), \quad z\in
\B_X\setminus\{0\}, \quad l_z\in T(z).
\]
Let $z=(z_1, z_2,0'')\in \B_X$ with $|z_1|=|z_2|=\rho \in (0,1)$.
Taking into account Lemma \ref{linear1}, we obtain that
\[
\frac{h_k(z)}{z_k}=1+
\sum_{\alpha\in \mathbb{N}^2, |\alpha|\geq 2}q_\alpha^k\frac{z^\alpha}{z_k}
\in g(\U),
\]
for $|z_k|=\|z\|=\rho\in (0,1)$, $k=1,2$.
Next,
let $\eta\in [0,2\pi)$ be such that
$q_{0,2}^1=|q_{0,2}^1|e^{i\eta}$, where $q_{0,2}^1=\frac{1}{2}\frac{\partial^2 h_1}{\partial z_2^2}(0)$.
For arbitrary $\theta,\lambda \in [0,2\pi)$, we put $z_1=\rho e^{i(\eta+\theta-\lambda)}$ and $z_2=\rho e^{i\theta/2}$.
Then, using the assumption that $g(\U)$ is convex, we deduce that
\[
\frac{1}{4\pi} \int_0^{4\pi}\frac{h_k(z)}{z_k}d\theta=
1+|q_{0,2}^1|e^{i\lambda}\rho \in \overline{g(\U)}.
\]
Since $\lambda \in [0,2\pi)$ is arbitrary, we have
$$|q_{0,2}^1|\rho\leq {\rm dist}(1,\partial g(\U))=d_1(g).$$
Letting
$\rho\to 1$,
we obtain that
\[
|q_{0,2}^1|\leq d_1(g).
\]

Next, we prove that $h_{1,2}^{[c]}\in {\mathcal M}_g(\B_X)$.
For $z\in \B_X\setminus \{ 0\}$,
let $l_z\in T(z)$ be arbitrarily fixed.
In view of Lemma \ref{estimate-trace},
we have
\[
|z_2|=\frac{|h_0(z,e_2)|}{|h_0(e_2,e_2)|}\leq \| z\|.
\]
Therefore, by
(\ref{estim1}),
we obtain that
\[
\left|\frac{1}{\| z\|}l_z(h_{1,2}^{[c]}(z))-1\right|
\leq \left|\frac{q_{0, 2}^1z_2^2}{\| z\|}\right|
<  d_1(g).
\]
Thus, we obtain that $h_{1,2}^{[c]}(z)\in  g(\U)$. Since
$z\in\B_X\setminus\{0\}$ and $l_z\in T(z)$ are arbitrary, we conclude that
$h_{1,2}^{[c]}\in {\mathcal M}_g(\B_X)$, as desired. This completes the proof.
\end{proof}

The following result yields that the estimate (\ref{estim1}) is sharp (cf. \cite{Br} and
\cite{GHKK16}, in the case of the Euclidean unit ball $\B^2$ in $\C^2$;
see also \cite[Proposition 4.6]{GHK18}, in the case of the unit bidisc
$\U^2$ in $\C^2$).

\begin{proposition}
\label{p.estim}
Let $\B_X$ be as in Definition \ref{d.shearing}.
Let $g:\U\to\C$
satisfy the conditions of Assumption
$\ref{assump}$ and let $d_1(g)={\rm dist}(1,\partial g(\U))$.
Let $h_{i,j}[g]:\B_X\to X$ be given by
$$h_{i,j}[g](z)=z\pm d_1(g)z_j^2e_i,\quad z\in\B_X,
\quad 1\leq i\neq j\leq r.$$
Then $h_{i,j}[g]\in {\mathcal M}_g(\B_X)$ for $1\leq i\neq j\leq r$.
Also, if $f_{i,j}[g]=f_{i,j}[g](z,t):\B_X\times [0,\infty)\to X$
are given by
$$f_{i,j}[g](z,t)=e^t\big(z\mp d_1(g)z_j^2 e_i\big),
\quad z\in\B_X,\quad t\geq 0,$$
then $f_{i,j}[g](z,t)$ are $g$-Loewner chains
for $1\leq i\neq j\leq r$.
In particular, $F_{i,j}[g]\in S_g^0(\B_X)$ and $F_{i,j}[g]$ are $g$-starlike on $\B_X$,
where
\begin{equation}
\label{extr-map}
F_{i,j}[g](z)=z\mp d_1(g)z_j^2e_i,\quad z\in\B_X,
\end{equation}
for
$1\leq i\neq j\leq r$.
\end{proposition}

\begin{proof}
Let $i,j\in \mathbb{Z}$ with $1\leq i\neq j\leq r$ be fixed.
The fact that $h_{i,j}[g]\in {\mathcal M}_g(\B_X)$ follows from
the proof of Proposition \ref{p.shear1}.
Also, it is easy to see that $f_{i,j}[g](z,t)$ is a normal Loewner chain which
satisfies the Loewner differential equation
$$\frac{\partial f_{i,j}[g]}{\partial t}(z,t)=Df_{i,j}[g](z,t)h_{i,j}[g](z),\quad \forall\, (z,t)\in\B_X\times [0,\infty).$$
Since $h_{i,j}[g]\in {\mathcal M}_g(\B_X)$, it follows that $f_{i,j}[g](z,t)$ is a $g$-Loewner chain, and thus
$F_{i,j}[g]=f_{i,j}[g](\cdot,0)\in S_g^0(\B_X)$. Finally, since $f_{i,j}[g](z,t)=e^tF_{i,j}[g](z)$ is a $g$-Loewner chain,
it follows that $F_{i,j}[g]\in S_g^*(\B_X)$, as desired. This completes the proof.
\end{proof}

Next, we point out certain particular cases of interest.
First, let $g(\zeta)=\frac{1-\zeta}{1+\zeta}$, $\zeta\in\U$.
Then we have ${\mathcal M}_g(\B_X)={\mathcal M}(\B_X)$
and ${\rm dist}\,(1,\partial g(\U))=1$.
In view of Propositions \ref{p.shear1} and \ref{p.estim}, we obtain
the following consequence (compare \cite[Proposition 2.1 and Corollary 2.2]{Br} and
\cite{GHKK16}, in the case of the Euclidean unit ball $\B^2$;
see \cite[Corollary 4.8]{GHK18} and \cite[Lemma 5]{Por1}, in the
case of the unit bidisc $\U^2$).

\begin{corollary}
\label{c.carath1}
Let $\B_X$ be as in Definition \ref{d.shearing}.
Let $h=(h_1,h_2, \dots, h_n)\in {\mathcal M}(\B_X)$.
Then $h_{i,j}^{[c]}\in {\mathcal M}(\B_X)$
for $1\leq i\neq j\leq r$,
and the following estimates hold:
\[
\left|
\frac{1}{2}\frac{\partial^2 h_i}{\partial z_j^2}(0)
\right|
\leq 1,
\quad
\mbox{for }
1\leq i\neq j\leq r.
\]
These estimates are sharp, for all $1\leq i\neq j\leq r$.
\end{corollary}

Next, let $\alpha\in [0,1)$ and $g(\zeta)=\frac{1-\zeta}{1+(1-2\alpha)\zeta}$,
$\zeta\in\U$. Then ${\mathcal M}_g(\B_X)={\mathcal M}_\alpha(\B_X)$
and ${\rm dist}(1,\partial g(\U))=d_1(\alpha)$, where (cf. \cite{GHKK16})
\begin{equation}
\label{a0-alpha}
d_1(\alpha)=\left\{\begin{array}{l} 1,\quad\quad\quad\quad \alpha\in [0, \frac{1}{2}]\\
\frac{1-\alpha}{\alpha},\quad\quad\quad\alpha\in (\frac{1}{2},1).
\end{array}\right.
\end{equation}
In this case, we obtain the following
sharp estimate for the family ${\mathcal M}_\alpha(\B_X)$ (compare with \cite{Br} and \cite{GHKK16},
in the case of the unit ball $\B^2$; see \cite{GHK18}, in the case of the
unit bidisc $\U^2$; see also \cite{S2} for $\alpha=0$ and $\B_X=\U^2$).

\begin{corollary}
\label{c.carath2}
Let $\B_X$ be as in Definition \ref{d.shearing}.
Let $\alpha\in [0,1)$ and let $h=(h_1,h_2, \dots, h_n)\in {\mathcal M}_\alpha(\B_X)$. Then
$h_{i,j}^{[c]}\in {\mathcal M}_\alpha(\B_X)$
for $1\leq i\neq j\leq r$, and the following estimates hold:
\[
\left|
\frac{1}{2}\frac{\partial^2 h_i}{\partial z_j^2}(0)
\right| \leq d_1(\alpha),
\quad
\mbox{for }
1\leq i\neq j\leq r,
\]
where $d_1(\alpha)$ is given
by $(\ref{a0-alpha})$.
These estimates are sharp, for all $i\neq j$.
\end{corollary}

Further, let $\alpha\in [0,1)$ and $g(\zeta)=\frac{1-(1-2\alpha)\zeta}{1+\zeta}$,
$\zeta\in \U$. Then, ${\rm dist}(1,\partial g(\U))=1-\alpha$.
From Propositions \ref{p.shear1} and \ref{p.estim},
we obtain the following consequence (compare \cite{Br} and \cite{GHKK16},
in the case of the unit ball $\B^2$; see \cite{GHK18}, in the case of the
unit bidisc $\U^2$).

\begin{corollary}
\label{c.carath3}
Let $\B_X$ be as in Definition \ref{d.shearing}.
Let $\alpha\in [0,1)$ and $h\in {\mathcal M}_g(\B_X)$,
where $g(\zeta)=\frac{1-(1-2\alpha)\zeta}{1+\zeta}$, $\zeta\in \U$.
Then $h_{i,j}^{[c]}\in {\mathcal M}_g(\B_X)$
for $1\leq i\neq j\leq r$, and the following sharp estimates hold:
\[
\left|
\frac{1}{2}\frac{\partial^2 h_i}{\partial z_j^2}(0)
\right|
\leq  1-\alpha,
\quad
\mbox{for }
1\leq i\neq j\leq r.
\]
\end{corollary}

If $\alpha\in (0,1]$ and $g(\zeta)=\left(\frac{1-\zeta}{1+\zeta}\right)^\alpha$, $\zeta\in\U$, where we choose
the branch of the power function such that $g(0)=1$, then we obtain the following coefficient bounds for
the family ${\mathcal M}_g(\B_X)$. In this case,
${\rm dist}\,(1,\partial g(\U))=\sin(\frac{\alpha\pi}{2})$.

\begin{corollary}
\label{c.carath-strongly-star}
Let $\B_X$ be as in Definition \ref{d.shearing}.
Let $\alpha\in (0,1]$ and let $h\in {\mathcal M}_g(\B_X)$,
where $g(\zeta)=\left(\frac{1-\zeta}{1+\zeta}\right)^\alpha$, $\zeta\in\U$.
Then $h_{i,j}^{[c]}\in {\mathcal M}_g(\B_X)$
for $1\leq i\neq j\leq r$, and the following sharp estimates hold:
\[\left|\frac{1}{2}\frac{\partial^2 h_i}{\partial z_j^2}(0)
\right|\leq\sin\left(\frac{\alpha\pi}{2}\right),\quad\mbox{ for } 1\leq i\neq j\leq r.\]
\end{corollary}

Next, we point out some particular cases of Proposition \ref{p.estim} and Corollaries
\ref{c.carath1}, \ref{c.carath2}, \ref{c.carath3}, and \ref{c.carath-strongly-star},
which provide examples of $g$-starlike mappings on $\B_X$ (cf. \cite[Lemma 2.3]{Br}
and \cite{GHKK16}, in the case of the unit ball $\B^2$; see \cite{GHK18}
for $\B_X=\U^2$).

\begin{corollary}
\label{c.estim2}
Let $\B_X$ be as in Definition \ref{d.shearing}. The following statements hold:

$(i)$ Let $\alpha\in [0,1)$ and $\Phi_{i,j}^{\alpha}:\B_X\to X$ be given by
\begin{equation}
\label{support2}
\Phi_{i,j}^{\alpha}(z)=z\pm d_1(\alpha)z_j^2 e_i,\quad z=(z_1,z_2, \dots, z_n)\in\B_X,
\end{equation}
for
$1\leq i\neq j\leq r$,
where $d_1(\alpha)$ is given by $(\ref{a0-alpha})$.
Then $\Phi_{i,j}^{\alpha}\in S_\alpha^*(\B_X)$, and thus
$\Phi_{i,j}^{\alpha}\in S_\alpha^0(\B_X)$
for
$1\leq i\neq j\leq r$.

$(ii)$ Let $\alpha\in [0,1)$ and $\Psi_{i,j}^{\alpha}:\B_X\to X$ be given by
\begin{equation}
\label{support3}
\Psi_{i,j}^{\alpha}(z)=z\pm (1-\alpha)z_j^2 e_i,\quad z=(z_1,z_2, \dots, z_n)\in\B_X,
\end{equation}
for
$1\leq i\neq j\leq r$.
Then $\Psi_{i,j}^{\alpha}\in {\mathcal A}S_\alpha^*(\B_X)$, and thus
$\Psi_{i,j}^{\alpha}\in {\mathcal A}S_\alpha^0(\B_X)$
for $1\leq i\neq j\leq r$.

$(iii)$ Let $\alpha\in (0,1]$ and $\Theta_{i,j}^{\alpha}:\B_X\to X$ be given by
\begin{equation}
\label{support3-strongly-star}
\Theta_{i,j}^{\alpha}(z)=z\pm\sin\Big(\frac{\alpha\pi}{2}\Big)z_j^2 e_i,
\quad z=(z_1,z_2, \dots, z_n)\in\B_X,
\end{equation}
for
$1\leq i\neq j\leq r$.
Then $\Theta_{i,j}^{\alpha}\in SS_\alpha^*(\B_X)$, and thus
$\Theta_{i,j}^{\alpha}\in SS_\alpha^0(\B_X)$
for $1\leq i\neq j\leq r$.
\end{corollary}

In view of the above results, it is natural to ask the following question:

\begin{question}
\label{q1}
Let $\B_X$ be as in Definition \ref{d.shearing}.
Let $g:\U\to \C$ be a function which satisfies the conditions of Assumption
$\ref{assump}$. Let $h=(h_1,h_2,\dots, h_n)\in {\mathcal M}_g(\B_X)$.
Is it
possible to find coefficient bounds for
$\big|\frac{1}{2}\frac{\partial^2 h_i}{\partial z_i^2}(0)\big|$, $1\leq i\leq r$
or
 for
$\Big|\frac{\partial^2 h_i}{{\partial z_i}{\partial z_j}}(0)\Big|$, $1\leq i\neq j\leq r$?
\end{question}

In the following, we give an answer to the above question
in the case
$g:\U\to\C$ is a univalent holomorphic function with $g(0)=1$ and $\Re g(\z)>0$
for $\z\in\U$.
(\ref{Cara-coefficient-12b})
is a generalization of \cite[Proposition 5.1 (b)]{S2}
(cf. \cite[Proposition 4.13]{GHK18}, \cite[Theorem 3.6]{GHHKS};
compare \cite{BR}, \cite[Lemma 5]{Por1}, for
$g(\zeta)=\frac{1-\zeta}{1+\zeta}$, $\zeta\in\U$).

\begin{proposition}
\label{p.coeff-mg}
Let $\B_X$ be as in Definition \ref{d.shearing}.
Let $g:\U\to\C$ be a univalent holomorphic function with $g(0)=1$ and $\Re g(\z)>0$
for $\z\in\U$, and let
$h=(h_1,\dots, h_n)\in {\mathcal M}_g(\B_X)$.
Then
\begin{equation}
\label{coeff-mg1}
\left|
\frac{1}{2}\frac{\partial^2 h_i}{\partial z_i^2}(0)
\right|
\leq |g'(0)|,
\quad
\mbox{for } 1\leq i\leq r
\end{equation}
and
\begin{equation}
\label{Cara-coefficient-12b}
\left|\frac{\partial^2 h_i}{{\partial z_i}{\partial z_j}}(0)\right|\leq |g'(0)|,\quad
1\leq i\neq j\leq r.
\end{equation}
These estimates are sharp for all $i,j $ with $1\leq i\neq j\leq r$.
\end{proposition}

\begin{proof}
Fix $w\in X$ such that $\|w\|=1$ and let $l_w\in T(w)$. Also,
let $p:\U\to\C$ given by
$p(\zeta)=\frac{1}{\zeta}l_w(h(\zeta w))$, $0<|\zeta|<1$, and $p(0)=1$.
Then $p\in H(\U)$, and since $h\in {\mathcal M}_g(\B_X)$, it is clear
that $p\prec g$. Consequently, we have that $|p'(0)|\leq |g'(0)|$. On the other hand,
it is not difficult to deduce that $p'(0)=\frac{1}{2}l_w(D^2h(0)(w,w))$, and thus
\begin{equation}
\label{coeff-mg3a}
\left|\frac{1}{2}l_w\big(D^2h(0)(w^2)\big)\right|\leq |g'(0)|.
\end{equation}

We may assume that $r=2$.
Also, let $v=(v_1,v_2,0'')\in X$ be such that $\|v\|=1$. Taking into
account Lemma \ref{linear1}, the relation (\ref{coeff-mg3a}), and the
fact that
\[
D^2h_i(0)(v,\cdot)=
\left(
\sum_{m=1}^2\frac{\partial^2 h_i}
{{\partial z_1}{\partial z_m}}(0)v_m,
\sum_{m=1}^2\frac{\partial^2 h_i}
{{\partial z_2}{\partial z_m}}(0)v_m,
\dots \right),
\]
for $i=1,2$,
we deduce that
\begin{equation}
\label{absolute}
\left|\frac{1}{2}\frac{\partial^2 h_i}{\partial z_1^2}(0)v_1^2+\frac{\partial^2 h_i}
{{\partial z_1}{\partial z_2}}(0)v_1v_2+\frac{1}{2}\frac{\partial^2 h_i}{\partial z_2^2}(0)v_2^2
\right|\leq |g'(0)|.
\end{equation}
From the above relation, we obtain (\ref{coeff-mg1}), as desired.
Also, from (\ref{absolute}), we have
\begin{equation}
\label{real}
\Re \left\{\frac{1}{2}\frac{\partial^2 h_i}{\partial z_1^2}(0)v_1^2+\frac{\partial^2 h_i}
{{\partial z_1}{\partial z_2}}(0)v_1v_2+\frac{1}{2}\frac{\partial^2 h_i}{\partial z_2^2}(0)v_2^2
\right\}\leq |g'(0)|.
\end{equation}
Let $\theta\in [0,2\pi)$ be such that
$e^{i\theta}\frac{\partial^2 h_i}
{{\partial z_1}{\partial z_2}}(0)=\left| \frac{\partial^2 h_i}
{{\partial z_1}{\partial z_2}}(0)\right|$.
Substituting $v_1=e^{i(\theta+\eta)}$ and $v_2=e^{-i\eta}$
into (\ref{real}) and integrating on $\eta\in [0,2\pi]$, we obtain (\ref{Cara-coefficient-12b})
as desired.

Finally, we prove that the relations (\ref{coeff-mg1}) and (\ref{Cara-coefficient-12b}) are sharp.
To this end, let $v=(v_1,v_2,0'')\in X$ be such that
$\|v\|=1$, and let $l_v\in T(v)$.
Also, let $h_v(z)=g(l_v(z))z$, $z\in\B_X$. Then $h_v\in {\mathcal M}_g(\B_X)$,
and after elementary computations, we obtain that
$\frac{1}{2}D^2h_v(0)(v^2)=g'(0)v$.
Next, if $v=(1,0,0'')$, then
$\frac{1}{2}\frac{\partial^2 h_1}{\partial z_1^2}(0)=g'(0)$,
while if $v=(0,1,0'')$,
then $\frac{1}{2}\frac{\partial^2 h_2}{\partial z_2^2}(0)=g'(0)$.
On the other hand,
let $H(z)=g(z_2)z$, $z\in\B_X$. Then $H\in {\mathcal M}_g(\B_X)$,
and after elementary computations, we obtain
\[
\frac{\partial^2 H_1}{\partial z_1\partial z_2}(0)=g'(0).
\]
The sharpness for $\frac{\partial^2 h_2}{\partial z_1\partial z_2}(0)$
is similar.
This completes the proof.
\end{proof}

\section{Support points for the family $S_g^0(\B_X)$}
\label{sec:support}
\setcounter{equation}{0}
\subsection{Bounded support points for $S_g^0(\B_X)$}
In the first part of this section we obtain sharp coefficient bounds for the family
$S_g^0(\B_X)$, where $g:\U\to\C$ is a convex (univalent) function on $\U$ which
satisfies the conditions of Assumption \ref{assump},
and $\B_X$ is the open unit ball
of an $n$-dimensional JB$^*$-triple $X=(\C^n, \|\cdot\|)$ of rank $r\geq 2$.
For particular choices of the function $g$, we obtain sharp coefficient bounds
for various subsets of $S^0(\B_X)$. Also, we obtain examples of bounded support
points for the family $S_g^0(\B_X)$.

This part is a continuation of recent works
on bounded support points for $S_g^0(\B^2)$ (see \cite{GHKK16}), and
in the case of the unit bidisc $\U^2$ of $\C^2$ (see \cite{GHK18}, \cite{S2}).

The following result is a generalization of \cite[Theorem 1.4]{Br}
and \cite[Theorem 4.10]{GHKK16} to the case
of $g$-Loewner chains on $\B_X\times [0,\infty)$, where
$g:\U\to\C$ satisfy the conditions of Assumption \ref{assump}.
We omit the proof of Theorem \ref{t.shearing-chain}, since it
suffices to use arguments similar to those in the proof of \cite[Theorem 4.10]{GHKK16}.

\begin{theorem}
\label{t.shearing-chain}
Let $\B_X$ be as in Definition \ref{d.shearing}.
Let $g:\U\to\C$
satisfy the conditions of Assumption
$\ref{assump}$.
Also, let $f(z,t):\B_X\times [0,\infty)\to
X$ be a $g$-Loewner chain. Then $f_{i,j}^{[c]}(z,t)$ are also $g$-Loewner chains, for
$1\leq i\neq j\leq r$.
In particular, if $f\in S_g^0(\B_X)$, then $f_{i,j}^{[c]}\in S_g^0(\B_X)$,
for $1\leq i\neq j\leq r$.
\end{theorem}

Next, we prove the following sharp coefficient bounds for the family $S_g^0(\B_X)$, where
$g:\U\to\C$ satisfies the conditions of Assumption \ref{assump}
(compare \cite[Theorem 3.1]{Br} and \cite{GHKK16}, in the case of the unit ball $\B^2$ in $\C^2$;
cf. \cite[Theorem 5.2]{GHK18}, in the case of the unit bidisc $\U^2$).
Other coefficient bounds for mappings
with $g$-parametric representation may be found in \cite{XL09}.

\begin{theorem}
\label{p.extremal}
Let $\B_X$ be as in Definition \ref{d.shearing}.
Let $g:\U\to\C$ satisfy the conditions of Assumption
$\ref{assump}$ and let $d_1(g)={\rm dist}(1,\partial g(\U))$.
Let $f=(f_1,\dots, f_n)\in S_g^0(\B_X)$.
Then
\[
\left|
\frac{1}{2}\frac{\partial^2 f_i}{\partial z_j^2}(0)
\right|
\leq d_1(g),
\quad
\mbox{for }
1\leq i\neq j\leq r.
\]
These estimates are sharp, for all $i\neq j$.
\end{theorem}

\begin{proof}
We only give a proof in the case $i=1$ and $j=2$.
The other cases can be proved by using similar arguments.
We prove that $|a_{2}^1|\leq d_1(g)$,
where $a_{2}^1=\frac{1}{2}\frac{\partial^2 f_1}
{\partial z_2^2}(0)$.
To this end,
we use some arguments similar to those in the proofs of \cite[Theorem 3.1]{Br}
and \cite[Theorem 4.11]{GHKK16}.
Since $f\in S_g^0(\B_X)$, there exists a $g$-Loewner
chain $f(z,t)$ such that $f=f(\cdot,0)$.
Let $h_t(z)=h(z,t)$ be the Herglotz vector field associated with $f(z,t)$.
In view of Theorem \ref{t.shearing-chain}, we deduce that the shearing $f_{1,2}^{[c]}(z,t)$
of $f(z,t)$ is a $g$-Loewner chain.
Then it is not difficult to deduce that $h_{1,2}^{[c]}(z,t)$ is the associated Herglotz
vector field of $f_{1,2}^{[c]}(z,t)$.
There exists a neighbourhood $V$ of the origin such that every holomorphic mapping on $V$
has a power series expansion on $V$.
Let
$$f_t(z)=e^{t}z+\cdots=\big(e^{t}z_1+\beta(t)z_2^2+\cdots,e^tz'+\cdots\big),\quad z\in V.$$
Then it is clear that $a_{2}^1=\beta(0)$. Also, let
$$h_t(z)=z+\cdots=\big(z_1+q(t)z_2^2+\cdots,z'+\cdots\big)$$
be the power series expansion of $h_t$ on $V$. Then
\begin{eqnarray*}
f_{1,2}^{[c]}(z,t)&=&
(e^{t}z_1+\beta(t)z_2^2,e^tz'),\quad z=(z_1,z_2, z'')\in V,\quad t\geq 0,\\
h_{1,2}^{[c]}(z,t)&=&
(z_1+q(t)z_2^2,z'), \quad z=(z_1,z_2, z'')\in V,\quad t\geq 0.
\end{eqnarray*}
Since $h_{1,2}^{[c]}(z,t)$ is the Herglotz vector field of $f_{1,2}^{[c]}(z,t)$, we deduce that
$$\frac{\partial f_{1,2}^{[c]}}{\partial t}(z,t)=Df_{1,2}^{[c]}(z,t)h_{1,2}^{[c]}(z,t),
\quad \mbox{ a.e. }\, t\geq 0,\quad \forall\, z\in V.$$
Identifying the coefficients in both sides of the above equality, we obtain that
$$\beta'(t)-2\beta(t)=e^tq(t),\quad \mbox{ a.e. }\, t\geq 0.$$
Therefore, we obtain
\[
\frac{d}{dt}\left(e^{-2t}\beta(t)\right)=e^{-t}q(t),
\quad \mbox{ a.e. }\quad t\geq 0.
\]
Integrating both sides of the above equality from $0$ to $t$, and using the fact that
$\beta(0)=a_{2}^1$, we deduce that
\begin{equation}
\label{estimate-integral}
e^{-2t}\beta(t)-a_{2}^1=\int_0^te^{-\tau}q(\tau)d\tau,\quad t\geq 0.
\end{equation}
On the other hand, in view of Proposition \ref{p.shear1},
we deduce that $|q(t)|\leq d_1(g)$, for a.e. $t\geq 0$.
Hence, in view of (\ref{estimate-integral}),
we deduce that
\begin{equation}
\label{estimate-integral2}
\left|e^{-2t}\beta(t)-a_{2}^1\right|\leq d_1(g)(1-e^{-t}),\quad t\geq 0.
\end{equation}
Next, since $\beta(t)=\frac{1}{2}\frac{\partial^2 f_{1,2}^{[c]}}{\partial z_2^2}(0,t)$ and
$\{e^{-t}f_{1,2}^{[c]}(\cdot,t)\}_{t\geq 0}$ is a normal family on $\B_X$, we obtain that
$\lim_{t\to\infty}e^{-2t}\beta(t)=0$, by using an argument similar to that in the proof
of \cite[Theorem 4.11]{GHKK16}.
Letting $t\to\infty$ in (\ref{estimate-integral2}), we deduce that
$|a_{2}^1|\leq d_1(g)$, as desired.
Sharpness of this relation is provided by
the mapping $F_{1,2}[g]\in S_g^0(\B_X)$ given by (\ref{extr-map}).
This completes the proof.
\end{proof}

From Theorem \ref{p.extremal}, we obtain the following sharp coefficient bounds for various subsets
of $S^0(\B_X)$
(see \cite{GHK18}, in the case $\B_X=\U^2$; see also \cite[Theorem 4.3]{S2}, in the case $\alpha=0$
and $\B_X=\U^2$; cf. \cite{Br}, \cite{BR},
\cite{GHKK16}, in the case of the Euclidean unit ball $\B^2$).

\begin{corollary}
\label{c.alpha}
Let $\B_X$ be as in Definition \ref{d.shearing}. The following statements hold:

$(i)$ Let $\alpha\in [0,1)$ and let
$f=(f_1,f_2, \dots, f_n)\in S_\alpha^0(\B_X)$.
Also, let $d_1(\alpha)$ be given by
$(\ref{a0-alpha})$. Then
\[
\left|
\frac{1}{2}\frac{\partial^2 f_i}{\partial z_j^2}(0)
\right| \leq d_1(\alpha),
\quad
\mbox{for }
1\leq i\neq j\leq r.
\]
These estimates are sharp, for all $i\neq j$.

In particular, if $f\in S^0(\B_X)$, then the following sharp
estimates hold:
\[
\left|
\frac{1}{2}\frac{\partial^2 f_i}{\partial z_j^2}(0)
\right| \leq 1,
\quad
\mbox{for }
1\leq i\neq j\leq r.
\]

$(ii)$ 
Let $\alpha\in [0,1)$ and let
$f=(f_1,f_2, \dots, f_n)\in {\mathcal A}S_\alpha^0(\B_X)$.
Then
\[
\left|
\frac{1}{2}\frac{\partial^2 f_i}{\partial z_j^2}(0)
\right|
\leq 1-\alpha,
\quad
\mbox{for }
1\leq i\neq j\leq r.
\]
These estimates are sharp, for all $i\neq j$.

$(iii)$ 
Let $\alpha\in (0,1]$and let
$f=(f_1,f_2, \dots, f_n)\in SS_\alpha^0(\B_X)$.
Then
\[
\left|
\frac{1}{2}\frac{\partial^2 f_i}{\partial z_j^2}(0)
\right|
\leq \sin\left(\frac{\alpha\pi}{2}\right),
\quad \mbox{ for } 1\leq i\neq j\leq r.
\]
These estimates are sharp, for all $i\neq j$.
\end{corollary}

Taking into account Theorem \ref{p.extremal}, we obtain the following result,
which provides examples of bounded support points for the families $S_g^0(\B_X)$
and $S_g^*(\B_X)$
(see \cite{GHK18} and \cite{S2}, in the
case of the unit bidisc $\U^2$;
compare \cite{Br}, \cite{GHKK15}, \cite{GHKK16}, in the case of the unit ball $\B^2$).

\begin{theorem}
\label{r.supp-bounded}
Let $\B_X$ be as in Definition \ref{d.shearing}.
Let $g:\U\to\C$ satisfy the conditions of Assumption $\ref{assump}$ and let
$d_1(g)={\rm dist}(1,\partial g(\U))$. Let
\begin{equation}
\label{supp3}
F_{i,j}[g](z)=z+d_1(g)z_j^2 e_i,\quad z\in\B_X,
\end{equation}
for $1\leq i\neq j\leq r$.
Then $F_{i,j}[g]\in S_g^*(\B_X)$ are bounded support points for $S_g^0(\B_X)$.
In particular, $F_{i,j}[g]$ are also bounded support points for $S_g^*(\B_X)$,
for $1\leq i\neq j\leq r$.
\end{theorem}

\begin{proof}
Let $i,j\in \mathbb{Z}$ with $1\leq i\neq j\leq r$ be fixed
and let $L_{i,j}: H(\B_X)\to\C$ be given by
$$L_{i,j}(f)=\frac{1}{2}\frac{\partial^2 f_i}{\partial z_j^2}(0),\quad
f=(f_1,f_2, \dots, f_n)\in H(\B_X).$$
Then $\Re L_{i,j}$ is a continuous linear functional on $H(\B_X)$,
and in view of Theorem \ref{p.extremal}, we obtain that
$$\Re L_{i,j}(q)\leq d_1(g)=\Re L(F_{i,j}[g]),\quad q\in S_g^0(\B_X),$$
where $F_{i,j}[g]\in S_g^*(\B_X)\subset S_g^0(\B_X)$ is given by (\ref{supp3}).
Since $\Re L_{i,j}({\rm id}_{\B_X})=0<d_1(g)=\Re L(F_{i,j}[g])$ and
${\rm id}_{\B_X}\in S_g^0(\B_X)$,
it follows that $\Re L_{i,j}|_{S_g^0(\B_X)}$ is not constant.
Hence $F_{i,j}[g]$ is a support point for the family $S_g^0(\B_X)$.
It is clear that the mapping $F_{i,j}[g]$ is bounded on $\B_X$.
Also, since $F_{i,j}[g]\in S_g^*(\B_X)$, the above arguments imply that
$F_{i,j}[g]$ is also a support point for the family $S_g^*(\B_X)$, as desired.
\end{proof}

\begin{remark}
Let $\B_X$ be as in Definition \ref{d.shearing}.
Let $g:\U\to\C$ satisfy the conditions of Assumption $\ref{assump}$ and let
$d_1(g)={\rm dist}(1,\partial g(\U))$.
Using arguments similar to those in the
proof of Theorem \ref{r.supp-bounded}, we deduce that if
$$\widehat{F}_{i,j}[g](z)=z-d_1(g)z_j^2 e_i,\quad z\in\B_X,$$
for $1\leq i\neq j\leq r$, then
$\widehat{F}_{i,j}[g]$ are bounded support points for $S_g^0(\B_X)$
and $S_g^*(\B_X)$.
\end{remark}

From Theorem \ref{r.supp-bounded} we obtain the following consequence
(see \cite[Remark 5.6]{GHK18}, in the case $\B_X=\U^2$; see \cite{S2}, for
$\alpha=0$; cf. \cite{Br} and \cite{GHKK16}, in the case of the ball $\B^2$ in $\C^2$).

\begin{corollary}
\label{cor.supp-bounded}
Let $\B_X$ be as in Definition \ref{d.shearing}. The following statements hold:

$(i)$ Let $\alpha\in [0,1)$ and let $\Phi_{i,j}^{\alpha}$ be the mapping given by 
$(\ref{support2})$.
Then $\Phi_{i,j}^{\alpha}$ are bounded support points for the family $S_\alpha^0(\B_X)$,
for $1\leq i\neq j\leq r$. Moreover, $\Phi_{i,j}^{\alpha}$
are also bounded support points for the family $S_\alpha^*(\B_X)$, for $1\leq i\neq j\leq r$.

$(ii)$ 
Let $\alpha\in [0,1)$ and let $\Psi_{i,j}^{\alpha}$ be the mapping given by $(\ref{support3})$.
Then $\Psi_{i,j}^{\alpha}$ are bounded support points for the family
$ {\mathcal A}S_\alpha^0(\B_X)$
for $1\leq i\neq j\leq r$. 
Moreover,
$\Psi_{i,j}^{\alpha}$ are bounded support points for the family ${\mathcal A}S_\alpha^*(\B_X)$
for $1\leq i\neq j\leq r$.

In particular, if
\begin{equation}
\label{extremal1}
\Phi_{i,j}(z)=z+z_j^2 e_i,\quad z\in\B_X,
\end{equation}
then $\Phi_{i,j}$ are bounded support points for $S^0(\B_X)$ and $S^*(\B_X)$,
for $1\leq i\neq j\leq r$

$(iii)$ 
Let $\alpha\in (0,1]$ and let $\Theta_{i,j}^{\alpha}:\B_X\to X$ be given by
$(\ref{support3-strongly-star})$. Then $\Theta_{i,j}^\alpha$ are bounded support
points for the family $SS_\alpha^0(\B_X)$, for $1\leq i\neq j\leq r$.
Moreover,
$\Theta_{i,j}^\alpha$ are bounded support points for the family $SS_\alpha^*(\B_X)$ for
$1\leq i\neq j\leq r$.
\end{corollary}

In view of the above arguments, it would be interesting to give an answer to the
following question:

\begin{question}
\label{q.extreme1}
Let $\B_X$ be as in Definition \ref{d.shearing}.
Let $g:\U\to\C$ satisfy the conditions of Assumptions $\ref{assump}$. Let
$d_1(g)={\rm dist}(1,\partial g(\U))$.
Let $F_{i,j}[g]:\B_X\to\C^n$ be given by $(\ref{supp3})$. Is it true that
$F_{i,j}[g]\in {\rm ex}\,S_g^0(\B_X)$? In particular, if $\Phi_{i,j}$ is given by
$(\ref{extremal1})$, then is it true that $\Phi_{i,j}\in {\rm ex}\,S^0(\B_X)$?
\end{question}

Taking into account Theorem \ref{p.extremal}, it is natural to ask the following question:

\begin{question}
\label{q.coeff-estim}
Let $\B_X$ be as in Definition \ref{d.shearing}.
Let $g:\U\to \C$ satisfy the conditions of Assumption
$\ref{assump}$. Let $f=(f_1,f_2,\dots, f_n)\in S_g^0(\B_X)$.
Is it possible to find sharp coefficient bounds for
$\big|\frac{1}{2}\frac{\partial^2 f_i}{\partial z_i^2}(0)\big|$, $1\leq i\leq r$
or for
$\Big|\frac{\partial^2 f_i}{{\partial z_i}{\partial z_j}}(0)\Big|$, $1\leq i\neq j\leq r$?
\end{question}

In the following result, we give a positive answer to the above question
in the case
$g:\U\to\C$ is a univalent holomorphic function with $g(0)=1$ and $\Re g(\z)>0$
for $\z\in\U$.
(\ref{coeff-gparam1b})
is a generalization of \cite[Theorem 4.3 (2)]{S2}
(cf. \cite[Theorem 5.7]{GHK18}, \cite[Theorem 2.14]{GHK02}, \cite[Theorem 10]{GHK14};
cf. \cite{BR}, \cite[Theorem 3]{Por1},
for $g(\zeta)=\frac{1-\zeta}{1+\zeta}$, $\zeta\in\U$).

\begin{theorem}
\label{t.coeff2}
Let $\B_X$ be as in Definition \ref{d.shearing}.
Let $g:\U\to\C$ be a univalent holomorphic function such that $g(0)=1$ and
$\Re g(\zeta)>0$, $\zeta\in\U$.
Also, let $f\in S_g^0(\B_X)$. Then
\begin{equation}
\label{coeff-gparam1}
\left|
\frac{1}{2}\frac{\partial^2 f_i}{\partial z_i^2}(0)
\right|
\leq |g'(0)|,
\quad
\mbox{for } 1\leq i\leq r
\end{equation}
and
\begin{equation}
\label{coeff-gparam1b}
\left|
\frac{\partial^2 f_i}{{\partial z_i}{\partial z_j}}(0)\right|
\leq |g'(0)|, \quad 1\leq i\neq j\leq r.
\end{equation}
These estimates are sharp for all $i \neq j$.
\end{theorem}

\begin{proof}
Since $f\in S_g^0(\B_X)$, there is a Herglotz vector field $h:\B_X\times [0,\infty)\to X$
such that $h(\cdot,t)\in {\mathcal M}_g(\B_X)$ for a.e. $t\geq 0$, and
$f(z)=\lim_{t\to\infty}e^tv(z,t)$ locally uniformly on $\B_X$, where
$v(z,t)=e^{-t}z+\cdots$ is the unique locally Lipschitz
continuous solution on $[0,\infty)$ of the initial value problem
\begin{equation}
\label{loewner-ode}
\frac{\partial v}{\partial t}=-h(v,t)\quad \mbox{ a.e. }\quad
t\geq 0,\quad v(z,0)=z,
\end{equation}
for all $z\in \B_X$.

Let $v_t=v(\cdot,t)$ and let $Q(t)(w^2)=\frac{1}{2}D^2v_t(0)(w^2)$ for $t\geq 0$
and $w\in X$. Since $v(z,\cdot)$
is locally Lipschitz continuous on $[0,\infty)$ locally uniformly with respect
to $z\in\B_X$, it follows in view of the Cauchy integral formulas for vector valued
holomorphic mappings that $Q(t)(w^2)$ is also locally Lipschitz continuous on
$[0,\infty)$, and thus $Q(\cdot)(w^2)$ is differentiable a.e. on $[0,\infty)$
(see e.g. \cite{GHK02} and \cite[Chapter 8]{GK03}).
There exists a neighbourhood $V$ of the origin such that every holomorphic mapping on $V$
has a power series expansion on $V$.
In view of (\ref{loewner-ode}), we
obtain after elementary computations that
$$-e^{-t}z+\frac{d}{dt}Q(t)(z^2)+\cdots=-e^{-t}z-Q(t)(z^2)-\frac{e^{-2t}}{2}D^2h(0,t)(z^2)-\cdots,$$
for a.e. $t\geq 0$ and for all $z\in V$. Identifying the coefficients in the above power
series expansions, we deduce that
$$\frac{d}{dt}Q(t)(z^2)=-Q(t)(z^2)-\frac{e^{-2t}}{2}D^2h(0,t)(z^2),\quad a.e.\quad t\geq 0,\quad
\forall\, z\in V,$$
and thus
$$\frac{d}{dt}\big(e^tQ(t)(w^2)\big)=-\frac{e^{-t}}{2}D^2h(0,t)(w^2),\quad a.e.\quad t\geq 0,\quad
\forall\, w\in X,\, \|w\|=1.$$
Integrating both sides of the above equality from $0$ to $t>0$, and using the fact that $v(z,0)=z$,
and thus $Q(0)(w^2)=0$, for all $w\in X$, we deduce that
$$e^tQ(t)(w^2)=-\int_0^te^{-\tau}\frac{1}{2}D^2h(0,\tau)(w^2)d\tau,\quad w\in X.$$
Next, fix $w\in X$ with $\|w\|=1$, and let $l_w\in T(w)$. Taking into account the relation
(\ref{coeff-mg3a}), we obtain that
\begin{eqnarray*}
\left|e^tl_w\big(Q(t)(w^2)\big)\right|
&\leq &\int_0^te^{-\tau}\left|\frac{1}{2}l_w\big(D^2h(0,\tau)(w^2)\big)\right|d\tau\\
&\leq &|g'(0)|(1-e^{-t}).
\end{eqnarray*}
Since $\lim_{t\to\infty}e^tv(\cdot,t)=f$ locally uniformly on $\B_X$, we obtain
in view of the above relation that
(see also \cite[Theorem 2.14]{GHK02}; cf. \cite[Theorem 3]{Por1}, in the case $\B_X=\U^n$)
\begin{equation}
\label{second-coeff-estimate}
\left|\frac{1}{2}l_w\big(D^2f(0)(w^2)\big)\right|\leq |g'(0)|.
\end{equation}
Since $w\in X$, $\|w\|=1$, and $l_w\in T(w)$ are arbitrary, the relations
(\ref{coeff-gparam1}) and (\ref{coeff-gparam1b})
easily follow in view of Lemma \ref{linear1} and the above inequality
(see the proof of Proposition \ref{p.coeff-mg}).

Next, we prove the sharpness of (\ref{coeff-gparam1}).
We may assume that $i=1$.
To this end, let $h:\B_X\to X$ be given by
$h(z)=g(z_1)z$. Then $h\in {\mathcal M}_g(\B_X)$. Since $\Re g(\zeta)>0$, $\zeta\in\U$, it
follows that ${\mathcal M}_g(\B_X)\subseteq {\mathcal M}(\B_X)$ and
$S_g^0(\B_X)\subseteq S^0(\B_X)$.
On the other hand, there exists $f\in S_g^*(\B_X)$ such $[Df(z)]^{-1}f(z)=h(z)$, $z\in\B_X$
(see \cite{GHK14}). It is not difficult to deduce that
$\frac{1}{2}D^2f(0)(u^2)=-\frac{1}{2}D^2h(0)(u^2)$, $u=(u_1,u_2,0'')\in X$, and thus
$$\frac{1}{2}\frac{\partial^2f_1}{\partial z_1^2}(0)=-g'(0),$$
which yields that (\ref{coeff-gparam1}) is sharp, as desired.
Finally,
we prove the sharpness of (\ref{coeff-gparam1b}).
We may assume that $i=1$ and $j=2$.
To this end, let $h:\B_X\to X$ be given by
$h(z)=g(z_2)z$. Then $h\in {\mathcal M}_g(\B_X)$ and there exists $f\in S_g^*(\B_X)$ such $[Df(z)]^{-1}f(z)=h(z)$, $z\in\B_X$.
By using arguments similar to the above, we have
\[
\frac{\partial^2f_1}{\partial z_1\partial z_2}(0)=-g'(0),
\]
which yields that (\ref{coeff-gparam1b}) is sharp, as desired.
This completes the proof.
\end{proof}

\begin{remark}
\label{second-remark}
In view of the above proofs,
(\ref{coeff-mg3a}) and (\ref{second-coeff-estimate}) 
hold for the unit ball in $\C^n$ with respect to an 
arbitrary norm on $\C^n$ and univalent functions $g$ with $g(0)=1$ and 
$\Re g(\zeta)>0$, for all $\zeta \in \U$.
\end{remark}

\subsection{An unbounded support point for $S_g^0(B)$}
In the case $g:\U\to\C$ is a univalent holomorphic function such that $g(0)=1$, 
$g(\overline{\zeta})=\overline{g(\zeta)}$, $\Re g(\zeta)>0$, $\zeta\in\U$, 
and $g(\rho)=O(1-\rho)$ as $\rho\to 1-0$,
we obtain an unbounded support point for $S_g^0(B)$,
where $B$ is the unit ball of $\C^n$ with respect to an arbitrary norm on $\C^n$
(cf. \cite[Corollary 4.4]{S2},
in the case of $S^0(\U^n)$).

\begin{theorem}
\label{t.unbounded-sup}
Let  $B$ be the unit ball of $\C^n$ with respect to an arbitrary norm on $\C^n$.
Let $g:\U\to\C$ be a univalent holomorphic function such that $g(0)=1$,
$g(\overline{\zeta})=\overline{g(\zeta)}$, 
and 
$\Re g(\zeta)>0$, $\zeta\in\U$. 
Assume that
$g(\rho)=O(1-\rho)$ as $\rho\to 1-0$.
Then there exists an unbounded support point for $S_g^0(B)$.
\end{theorem}

\begin{proof}
First, note that $g(\zeta)$ is real valued for real $\zeta \in\U$ and
$g'(0)<0$ from the assumption.

Let $b\in S_g^*(\U)$ be defined by $b(0)=b'(0)-1=0$ and
\[
\frac{\zeta b'(\zeta)}{b(\zeta)}=\frac{1}{g(\zeta)},
\quad \zeta \in\U.
\]
Let $f^{e_1}:B\to X$ be given by
\begin{equation}
\label{star-g-map}
f^{e_1}(z)=\frac{b(l_{e_1}(z))}{l_{e_1}(z)}z,
\quad z\in B,
\end{equation}
where $e_1$ is a unit vector in $\C^n$.
Let $f_1(\zeta)=l_{e_1}(f(\zeta e_1))$ for $f\in H(B)$.
Then $f^{e_1}\in S^*_g(B)$
and
$\frac{1}{2}\frac{\partial ^2 f_1^{e_1}}{\partial \zeta^2}(0)=
-g'(0)=|g'(0)|$
by \cite[Lemma 3.1]{HH08}.
Let
\[
\Lambda(f)=\frac{1}{2}\frac{\partial ^2 f_1}{\partial \zeta^2}(0),
\quad
f\in H(B).
\]
Since $\Re \Lambda(f)\leq \Re \Lambda(f^{e_1})$
for $S_g^0(B)$ by Remark \ref{second-remark}
(see also \cite[Theorem 2.14]{GHK02}),
and since $\Re \Lambda({\rm id}_{B})=0<\Re \Lambda(f^{e_1})$,
it follows that $f^{e_1}$ is a support point for $S_g^0(B)$.

Since there exists a constant $C>0$ such that
\[
\frac{b'(\rho)}{b(\rho)}=\frac{1}{\rho g(\rho)}\geq \frac{C}{1-\rho},
\quad \frac{1}{2}<\rho<1,
\]
we have
\[
\log \frac{b(\rho)}{b(\frac{1}{2})}
\geq
-C\log 2(1-\rho)=\log\big\{2(1-\rho)\big\}^{-C},
\quad
\frac{1}{2}<\rho<1.
\]
Therefore, we have
\[
b(\rho)\geq b\left(\frac{1}{2}\right)\big\{2(1-\rho)\big\}^{-C},
\quad
\frac{1}{2}<\rho<1.
\]
This implies that
$\| f^{e_1}( \rho e_1)\|=\| b(\rho )e_1\|=b(\rho) \to \infty$
as $\rho\to 1-0$.
Thus, $f^{e_1}$ is an unbounded support point for $S_g^0(B)$.
This completes the proof.
\end{proof}

\begin{remark}
Let $\B_X$ be as in Definition \ref{d.shearing}. 

(i) In view of  Corollary \ref{cor.supp-bounded} and Theorem \ref{t.unbounded-sup},
we deduce that the family $S_\alpha^0(\B_X)$ contains bounded and also unbounded support
points, for $\alpha\in [0,1)$. In particular, the family $S^0(\B_X)$ contains bounded
and also unbounded support points.

(ii) Let $g:\U\to\C$ satisfy the conditions of Assumption $\ref{assump}$, and 
let $d_1(g)={\rm dist}(1,\partial g(\U))$.
We remark that the support points of $S_g^0(\B_X)$
given by (\ref{supp3}) are restriction to $\B_X$ of
automorphisms of $X=\C^n$. However, if
$g(\zeta)=\frac{1-A\zeta}{1+B\zeta}$, $\zeta\in\U$,
with $-1\leq A<1$, $A+B>0$, $-1\leq B \leq 1$,
then the mapping $f^{e_1}$ in the proof of Theorem \ref{t.unbounded-sup}, given by
(\ref{star-g-map}), is a bounded support point of $S_g^0(\B_X)$,
which cannot be extended to an automorphism of $\C^n$.
\end{remark}

Finally, in view of Theorem \ref{t.unbounded-sup}, we point out the following questions
of interest:

\begin{question}
\label{q.unbounded-ex1}
Let $B$ be the unit ball of $\C^n$ with respect to an 
arbitrary norm on $\C^n$.
Let $g:\U\to\C$ be a univalent holomorphic function such that $g(0)=1$, 
$g(\overline{\zeta})=\overline{g(\zeta)}$, 
and $\Re g(\zeta)>0$, $\zeta\in\U$. Assume that
$g(\rho)=O(1-\rho)$ as $\rho\to 1-0$. Does there exist
an unbounded extreme point of $S_g^0(B)$?
\end{question}

\begin{question}
\label{q.unbounded-ex2}
Let $B$ be the unit ball of $\C^n$ with respect to an 
arbitrary norm on $\C^n$.
Let $g:\U\to\C$ be as in Question $\ref{q.unbounded-ex1}$.
Also, let $b\in {\rm ex}\,S_g^*(\U)$, and let $f^{e_1}:B\to\C^n$
be given by $(\ref{star-g-map})$. Then is it true that $f^{e_1}\in
{\rm ex}\, S_g^*(B)$?
\end{question}

\section{Coefficient bounds and support points on the Euclidean unit ball}
\setcounter{equation}{0}
In the case of the Euclidean unit ball $\B^n$ of $\C^n$ with $n\geq 2$,
we point out the following results, which are generalizations of recent
results from \cite[Section 4]{GHKK16} (compare with \cite{Br} and
\cite{BR}, in the case $g(\zeta)=\frac{1-\zeta}{1+\zeta}$, $\zeta\in\U$).
$\B^n$ is a bounded symmetric domain of rank $r=1$.
We take $e_1=(1,0,\dots, 0), e_2=(0,1,\dots, 0), e_n=(0,0, \dots, 1)\in \C^n$ 
as the basis explained in the introduction.
Then, $z=(z_1,\dots, z_n)\in \C^n$ is the usual expression of $z\in \C^n$.
We omit the proofs, because it suffices to use arguments similar
to those in the previous sections.

\begin{proposition}
\label{p.shear1-E}
Let $g:\U\to \C$ satisfy the conditions of Assumption
$\ref{assump}$ and let $d_1(g)={\rm dist}(1,\partial g(\U))$.
Let $n\geq 2$.

$(i)$ If $h=(h_1,h_2,\dots, h_n)\in {\mathcal M}_g(\B^n)$, then
$$\left|
\frac{1}{2}\frac{\partial^2 h_i}{\partial z_j^2}(0)
\right| \leq  \frac{3\sqrt{3}}{2}d_1(g),
\quad
\mbox{for }
1\leq i\neq j\leq n.$$
In addition, $h_{i,j}^{[c]}\in {\mathcal M}_g(\B^n)$
for $1\leq i\neq j\leq n$.

$(ii)$
Let $\widetilde{h}_{i,j}[g]:\B^n\to \C^n$ be given by
$$\widetilde{h}_{i,j}[g](z)=z\pm  \frac{3\sqrt{3}}{2}d_1(g)z_j^2e_i,\quad z\in\B^n,
\quad 1\leq i\neq j\leq n.$$
Then $\widetilde{h}_{i,j}[g]\in {\mathcal M}_g(\B^n)$ are bounded support points of 
${\mathcal M}_g(\B^n)$ for $1\leq i\neq j\leq n$.
\end{proposition}

In particular, from Proposition \ref{p.shear1-E}, we obtain the following
consequences:

\begin{corollary}
\label{c.alpha-star-E}
Let $n\geq 2$, $\alpha\in [0,1)$, and let $h\in {\mathcal M}_g(\B^n)$,
where $g(\zeta)=\frac{1-\zeta}{1+(1-2\alpha)\zeta}$, $\zeta\in\U$.
Then $h_{i,j}^{[c]}\in {\mathcal M}_g(\B^n)$
for $1\leq i\neq j\leq n$, and the following sharp estimates hold:
\[\left|\frac{1}{2}\frac{\partial^2 h_i}{\partial z_j^2}(0)
\right|\leq
\frac{3\sqrt{3}}{2}d_1(\alpha),\quad\mbox{ for } 1\leq i\neq j\leq n,\]
where $d_1(\alpha)$ is given by $(\ref{a0-alpha})$.
\end{corollary}

\begin{corollary}
\label{c.carath-strongly-star-E}
Let $n\geq 2$, $\alpha\in (0,1]$, and let $h\in {\mathcal M}_g(\B^n)$,
where $g(\zeta)=\left(\frac{1-\zeta}{1+\zeta}\right)^\alpha$, $\zeta\in\U$,
and the branch of the power function is chosen such
that $g(0)=1$.
Then $h_{i,j}^{[c]}\in {\mathcal M}_g(\B^n)$
for $1\leq i\neq j\leq n$, and the following sharp estimates hold:
\[\left|\frac{1}{2}\frac{\partial^2 h_i}{\partial z_j^2}(0)
\right|\leq
\frac{3\sqrt{3}}{2}\sin\left(\frac{\alpha\pi}{2}\right),\quad\mbox{ for }
1\leq i\neq j\leq n.\]
\end{corollary}

\begin{theorem}
\label{t.estim-E},
Let $g:\U\to \C$ satisfy the conditions of Assumption
$\ref{assump}$ and let $d_1(g)={\rm dist}(1,\partial g(\U))$.
Let $n\geq 2$.

$(i)$
Let $f=(f_1,\dots, f_n)\in S_g^0(\B^n)$.
Then
\[
\left|
\frac{1}{2}\frac{\partial^2 f_i}{\partial z_j^2}(0)
\right|
\leq \frac{3\sqrt{3}}{2}d_1(g),
\quad
\mbox{for }
1\leq i\neq j\leq n.
\]
These estimates are sharp, for all $i\neq j$.

$(ii)$
If $\widetilde{f}_{i,j}[g]=\widetilde{f}_{i,j}[g](z,t):\B^n\times [0,\infty)\to \C^n$
are given by
$$\widetilde{f}_{i,j}[g](z,t)=e^t\left( z\mp  \frac{3\sqrt{3}}{2}d_1(g)z_j^2 e_i\right),
\quad z\in\B^n,\quad t\geq 0,$$
then $\widetilde{f}_{i,j}[g](z,t)$ are $g$-Loewner chains
for $1\leq i\neq j\leq n$.
In particular, $\widetilde{F}_{i,j}[g]\in S_g^0(\B^n)$ and $\widetilde{F}_{i,j}[g]$
are $g$-starlike on $\B^n$, where
\begin{equation}
\label{supp3-E}
\widetilde{F}_{i,j}[g](z)=z\mp\frac{3\sqrt{3}}{2}d_1(g)z_j^2e_i,\quad z\in\B^n,
\end{equation}
for $1\leq i\neq j\leq n$.

$(iii)$
$\widetilde{F}_{i,j}[g]\in S_g^*(\B^n)$ are bounded support points for $S_g^0(\B^n)$,
for $1\leq i\neq j\leq n$.
In particular, $\widetilde{F}_{i,j}[g]$ are also bounded support points for $S_g^*(\B^n)$,
for $1\leq i\neq j\leq n$.
\end{theorem}

In particular, from Theorem \ref{t.estim-E}, we obtain the following results:

\begin{corollary}
\label{c.estim2-E}
Let $n\geq 2$, $\alpha\in [0,1)$.

$(i)$
Let $f\in S_{\alpha}^0(\B^n)$.
Then
\[\left|\frac{1}{2}\frac{\partial^2 f_i}{\partial z_j^2}(0)\right|
\leq \frac{3\sqrt{3}}{2}d_1(\alpha),
\quad \mbox{ for } 1\leq i\neq j\leq n.\]
These estimates are sharp, for all $i\neq j$.

$(ii)$
Let
$\widetilde{\Phi}_{i,j}^{\alpha}:\B^n\to\C^n$ be given by
\begin{equation}
\label{support2b}
\widetilde{\Phi}_{i,j}^{\alpha}(z)=z\pm \frac{3\sqrt{3}}{2}
d_1(\alpha)z_j^2 e_i,\quad z=(z_1,z_2,\dots, z_n)\in\B^n,
\end{equation}
for $1\leq i\neq j\leq r$,
where $d_1(\alpha)$ is given by $(\ref{a0-alpha})$.
Then $\widetilde{\Phi}_{i,j}^{\alpha}\in S_{\alpha}^*(\B^n)$, and thus
$\widetilde{\Phi}_{i,j}^{\alpha}\in S_{\alpha}^0(\B^n)$,
for $1\leq i\neq j\leq r$.

$(iii)$
$\widetilde{\Phi}_{i,j}^{\alpha}\in S_{\alpha}^*(\B^n)$ are bounded support
points for the family $S_{\alpha}^0(\B^n)$
and also bounded support points for the
family $S_{\alpha}^*(\B^n)$, for $1\leq i\neq j\leq n$.
\end{corollary}

\begin{corollary}
\label{c.estim-strongly-star-E}
Let $n\geq 2$, $\alpha\in (0,1]$.

$(i)$
Let $f\in SS_{\alpha}^0(\B^n)$.
Then
\[
\left|
\frac{1}{2}\frac{\partial^2 f_i}{\partial z_j^2}(0)
\right|
\leq \frac{3\sqrt{3}}{2}\sin\left(\frac{\alpha\pi}{2}\right),
\quad \mbox{ for } 1\leq i\neq j\leq n.
\]
These estimates are sharp, for all $i\neq j$.

$(ii)$
Let
$\widetilde{\Theta}_{i,j}^{\alpha}:\B^n\to \C^n$ be given by
\begin{equation}
\label{support3-strongly-star-E}
\widetilde{\Theta}_{i,j}^{\alpha}(z)=z\pm\frac{3\sqrt{3}}{2}\sin
\Big(\frac{\alpha\pi}{2}\Big)z_j^2 e_i,\quad z=(z_1,z_2, \dots, z_n)\in\B^n,
\end{equation}
for $1\leq i\neq j\leq n$.
Then $\widetilde{\Theta}_{i,j}^{\alpha}\in SS_{\alpha}^*(\B^n)$, thus
$\widetilde{\Theta}_{i,j}^{\alpha}\in SS_{\alpha}^0(\B^n)$,
$1\leq i\neq j\leq n$.

$(iii)$
$\widetilde{\Theta}_{i,j}^\alpha\in SS_{\alpha}^*(\B^n)$
are bounded support
points for the family $SS_{\alpha}^0(\B^n)$
and also bounded support points for the
family $SS_{\alpha}^*(\B^n)$
for $1\leq i\neq j\leq n$.
\end{corollary}

\section*{Acknowledgments}
\label{acknowledgments}
{Hidetaka Hamada was partially supported by JSPS KAKENHI Grant Number JP19K03553.
We thank M. Iancu for discussions concerning Proposition \ref{p.a0}.}


\begin{thebibliography}{100}
\bibitem{ABW}
Arosio L,
Bracci F,
Wold E.F.
{Solving the Loewner PDE in
complete hyperbolic starlike domains of $\mathbb{C}^n$}.
Adv.  Math., 2013, {242}: 209--216

\bibitem{Br}
Bracci F.
{Shearing process and an example of a bounded support function
in $S^0(\mathbb{B}^2)$}.
Comput. Methods Funct. Theory,
2015, {15}: 151--157

\bibitem{BCM}
Bracci F,
Contreras M.D,
Madrigal S.-D\'iaz.
{Evolution families and the Loewner equation II:
complex hyperbolic manifolds}. Math. Ann., 2009, {344}: 947--962

\bibitem{BGHK-Variation}
Bracci F, 
Graham I,
Hamada H,
et al.
{Variation of Loewner chains,
extreme and support points in the class $S^0$ in higher dimensions},
Constructive Approx., 2016, {43}: 231--251

\bibitem{BR}
Bracci F, 
Roth O.
{Support points and the Bieberbach conjecture in higher
dimension}. 
In: Complex analysis and dynamical systems, 67--79,
Trends Math., Birkh\"auser-Springer, Cham, 2018

\bibitem{Car}
Cartan H. 
Sur la possibilit\'e d'\'etendre aux fonctions de plusieurs
variables complexes la th\'eorie des fonctions univalentes.
129--155. Note
added to P. Montel, Le\c cons sur les Fonctions Univalentes ou Multivalentes,
Gauthier-Villars, Paris, 1933

\bibitem{C12}
Chu C.-H.
{Jordan Structures in Geometry and Analysis}. In: Cambridge Tracts in Mathematics, 190,
Cambridge University Press, Cambridge, 2012


\bibitem{DGHK}
Duren P,
Graham I,
Hamada H,
et al.
{Solutions for the generalized
Loewner differential equation in several complex variables}. Math. Ann., 2010,
{347}: 411--435

\bibitem{ERS}
Elin M, 
Reich S,
Shoikhet D. 
{Complex dynamical systems and the geometry of domains in Banach spaces}.
Dissertationes Math., 2004, {427}: 1--62

\bibitem{Gon2}
Gong S.
{Convex and Starlike Mappings in Several Complex Variables}.
Kluwer Acad. Publ., Dordrecht, 1998

\bibitem{GHHKS}
Graham I,
Hamada H,
Honda T,
et al.
{Growth, distortion and coefficient bounds for Carath\'eodory
families in $\C^n$ and complex Banach spaces}.
J. Math. Anal. Appl., 2014, {416}: 449--469


\bibitem{GHK02}
Graham I,
Hamada H,
Kohr G.
{Parametric representation of univalent mappings
in several complex variables}.
Canadian J. Math., 2002, {54}: 324--351


\bibitem{GHK14}
Graham I,
Hamada H,
Kohr G.
{Extremal problems and $g$-Loewner chains in
$\C^n$ and reflexive complex Banach spaces}.
In: {Topics in Mathematical Analysis and Applications}
(eds. T.M. Rassias and L. Toth), Springer, Cham, 2014, {94}: 387--418

\bibitem{GHK18}
Graham I,
Hamada H,
Kohr G.
{Extremal problems for mappings with $g$-parametric representation on the unit
polydisc in $\C^n$}, In: Complex Analysis and Dynamical Systems (eds. M. Agranovsky et al.),
Birkh\"{a}user series Trends in Mathematics, 2018, 141--167


\bibitem{GHKK12}
Graham I,
Hamada H,
Kohr G,
et al.
{Extreme points, support
points and the Loewner variation in several complex variables}.
Sci. China Math., 2012, {55}: 1353--1366


\bibitem{GHKK13}
Graham I,
Hamada H,
Kohr G,
et al.
{Extremal properties associated with univalent
subordination chains in $\C^n$}. Math. Ann., 2014, {359}: 61--99


\bibitem{GHKK15}
Graham I,
Hamada H,
Kohr G,
et al.
{Support points and extreme
points for mappings with $A$-parametric representation in $\C^n$}.
J. Geom. Anal., 2016, {26}: 1560--1595

\bibitem{GHKK16}
Graham I,
Hamada H,
Kohr G,
et al.
{Bounded support points for mappings with $g$-parametric representation in $\C^2$}.
J. Math. Anal. Appl., 2017, {454}: 1085--1105

\bibitem{GK03}
Graham I,
Kohr G.
{Geometric Function Theory in One and Higher Dimensions}.
Marcel Dekker Inc., New York, 2003


\bibitem{H18}
Hamada H.
{A Schwarz lemma at the boundary using the Julia-Wolff-Carath\'eodory type condition
on finite dimensional irreducible bounded symmetric domains}.
J. Math. Anal. Appl., 2018, {465}: 196--210

\bibitem{HH08}
Hamada H,
Honda T.
{Sharp growth theorems and
coefficient bounds for starlike mappings in several complex
variables}.
Chinese Ann. Math. Ser. B, 2008, {29}: 353--368

\bibitem{HIK}
Hamada H,
Iancu M,
Kohr G.
{Extremal problems for mappings with generalized
parametric representation in $\C^n$}.
Complex Anal. Oper. Theory, 2016, {10}: 1045--1080

\bibitem{HS}
Hengartner W,
Schober G.
{On schlicht mappings to domains convex in one direction}.
Comment. Math. Helv.
{45}  (1970) 303--314.

\bibitem{K83}
{Kaup W.}
{A Riemann mapping theorem for bounded symmetric domains in complex Banach spaces}.
{Math. Z.}, 1983,
{183}: 503--529

\bibitem{LLX15}
Liu X.S,
Liu T.S,
Xu Q.H. 
A proof of a weak version of the Bieberbach conjecture in several complex variables.
Sci. China Math., 2015, {58}: 2531--2540

\bibitem{L}
Loos O.
Bounded symmetric domains and Jordan pairs.
University of California, Irvine,
1977

\bibitem{MaMi}
Ma W,
Minda D.
{A unified treatment of some special classes of
univalent functions}.
Proceedings of the Conference on Complex Analysis
(Tianjin, 1992), 157--169, Int. Press, Cambridge, MA, 1994.

\bibitem{M18}
Muir Jr. J.R.
{Convex families of holomorphic mappings related to the
convex mappings of the ball in $\C^n$}.
Proc. Amer. Math. Soc., 2019,
{147}: 2133--2145

\bibitem{Pf74}
Pfaltzgraff J.A.
{Subordination chains and univalence of
holomorphic mappings in $\mathbb{C}^n$}.
Math. Ann., 1974, {210}:
55--68

\bibitem{Pom2}
Pommerenke Ch.
{Univalent Functions}. Vandenhoeck \& Ruprecht,
G\"ottingen, 1975

\bibitem{Pom1}
Pommerenke Ch.
{Boundary Behaviour of Conformal Mappings}. Springer-Verlag, 1992

\bibitem{Por1}
Poreda T. 
{On the univalent holomorphic maps of the unit polydisc
in $\mathbb{C}^n$ which have the parametric representation, I-the
geometrical properties}.
Ann. Univ. Mariae Curie Skl. Sect. A.,
1987, {41}: 105--113

\bibitem{Por2}
Poreda T. 
{On the univalent holomorphic maps of the unit polydisc
in $\mathbb{C}^n$ which have the parametric representation, II-the
necessary conditions and the sufficient conditions}.
Ann. Univ. Mariae Curie Skl. Sect. A., 1987, {41}: 115--121

\bibitem{ReSh}
Reich S,
Shoikhet D.
{Nonlinear Semigroups, Fixed Points, and Geometry of Domains in
Banach Spaces}. Imperial College Press, London, 2005

\bibitem{RS}
Roper K,
Suffridge  T.J. 
{Convexity properties of holomorphic mappings in
$\mathbb{C}^n$}.
Trans. Amer. Math. Soc., 1999, {351}: 1803--1833


\bibitem{Roth-pontryagin}
Roth O. 
{Pontryagin's maximum principle for the Loewner equation
in higher dimensions}.
Canad. J. Math., 2015, {67}: 942--960

\bibitem{S}
Schleissinger S.
{On support points of the class $S^0(\B^n)$}.
Proc. Amer. Math. Soc., 2014, {142}: 3881--3887

\bibitem{S2}
Schleissinger S.
{On the parametric representation of univalent functions on the
polydisc}.
Rocky Mountain J. Math.,
2018, {48}: 981--1001

\bibitem{Su}
Suffridge  T.J.
{Starlikeness, convexity and other geometric properties
of holomorphic maps in higher dimensions}. In: Lecture Notes Math., Springer, Berlin,
{599}: 146--159, 1977

\bibitem{Vo}
Voda M.
{Solution of a Loewner chain equation in several complex
variables}.
J. Math. Anal. Appl., 2011, {375}: 58--74


\bibitem{XL}
Xu Q.H,
Liu T.S.
{L\"owner chains and a subclass of biholomorphic mappings}.
J. Math. Anal. Appl., 2007, {334}: 1096--1105

\bibitem{XL09SC}
Xu Q.H,
Liu T.S.
On coefficient estimates for a class of holomorphic mappings.
Sci. China Ser. A, 2009,  {52}: 677--686

\bibitem{XL09}
Xu Q.H,
Liu T.S.
{Coefficient bounds for biholomorphic mappings which have a parametric representation}.
J. Math. Anal. Appl., 2009, {355}: 126--130

\bibitem{XLL18}
Xu Q.H,
Liu T.S,
Liu X.S.
Fekete and Szeg\"o problem in one and higher dimensions.
Sci. China Math., 2018, {61}: 1775--1788
\end{thebibliography}
\end{document}